\UseRawInputEncoding
\documentclass[12pt,reqno]{amsart}
\usepackage{amssymb,latexsym,color,}
\usepackage{amsmath}
\usepackage{amsthm}
\usepackage{graphicx}
\usepackage{palatino,mathpazo}
\usepackage{titletoc}

\usepackage{times,dsfont,ifthen,mathrsfs,amsfonts}

by-\headheight \advance \topmargin by-\headsep 
\numberwithin{equation}{section}
\newtheorem{theorem}{Theorem}[section]
\newtheorem{proposition}[theorem]{Proposition}
\newtheorem{lemma}[theorem]{Lemma}
\newtheorem{remark}[theorem]{Remark}
\newtheorem{corollary}[theorem]{Corollary}

\newtheorem{definition}[theorem]{Definition}

\theoremstyle{definition}

\newtheorem{example}{Example}[section]

\def\XXint#1#2#3{{\setbox0=\hbox{$#1{#2#3}{\int}$}
     \vcenter{\hbox{$#2#3$}}\kern-.5\wd0}}

\newcommand{\R}{\mathbb{R}}
\newcommand{\RN}{\mathbb{R}^2}

\newcommand{\lab}{\label}
\newcommand{\bt}{\begin{theorem}}
\newcommand{\et}{\end{theorem}}
\newcommand{\bl}{\begin{lemma}}
\newcommand{\el}{\end{lemma}}
\newcommand{\bd}{\begin{definition}}
\newcommand{\ed}{\end{definition}}
\newcommand{\bc}{\begin{corollary}}
\newcommand{\ec}{\end{corollary}}
\newcommand{\bp}{\begin{proof}}
\newcommand{\ep}{\end{proof}}
\newcommand{\bx}{\begin{example}}
\newcommand{\ex}{\end{example}}
\newcommand{\bi}{\begin{exercise}}
\newcommand{\ei}{\end{exercise}}
\newcommand{\bo}{\begin{proposition}}
\newcommand{\eo}{\end{proposition}}
\newcommand{\brr}{\begin{remark}}
\newcommand{\er}{\end{remark}}
\newcommand{\beq}{\begin{equation}}
\newcommand{\eeq}{\end{equation}}
\newcommand{\ba}{\begin{align}}
\newcommand{\ea}{\end{align}}
\newcommand{\bn}{\begin{enumerate}}
\newcommand{\en}{\end{enumerate}}
\newcommand{\bg}{\begin{align*}}
\newcommand{\bcs}{\begin{cases}}
\newcommand{\ecs}{\end{cases}}

\newcommand{\De}{\Delta}

\begin{document}
\title[A global branch approach to normalized solutions]{A global branch approach to normalized solutions for the Schr\"odinger equation}
 \thanks{Xuexiu Zhong was supported by the NSFC (No.12271184), Guangdong Basic and Applied Basic Research Foundation (2021A1515010034), Guangzhou Basic and Applied Basic Research Foundation(202102020225),Guangdong Natural Science Fund (2018A030310082). Jianjun Zhang was supported by the NSFC (No.12371109)}
\author{ Louis Jeanjean}
\address{ Louis Jeanjean,~Laboratoire de Math\'ematiques (CNRS UMR 6623),  Universit\'e de Bourgogne Franche-Comt\'e, 16 Rte de Gray, 25030,Besancon, France.}
\email{louis.jeanjean@univ-fcomte.fr}
\author{Jianjun Zhang}
\address{ Jianjun Zhang,~College of Mathematics and Statistics, Chongqing Jiaotong University,Xuefu,Nan'an,400074, Chongqing, PR China.}
\email{zhangjianjun09@tsinghua.org.cn}
\author{Xuexiu Zhong}
\address{ Xuexiu ~Zhong,~South China Research Center for Applied Mathematics and Interdisciplinary Studies \& School of Mathematics Science,South China Normal University, Tianhe, Guangzhou 510631, Guangdong, PR China.}
\email{zhongxuexiu1989@163.com}
\date{}
\begin{abstract}
We study the existence, non-existence and multiplicity of prescribed mass positive solutions to a Schr\"odinger equation of the form
\begin{equation*}
-\Delta u+\lambda u=g(u), \quad u \in H^1(\mathbb{R}^N), \, N \geq 1.
\end{equation*}
Our approach permits to handle in a unified way nonlinearities $g(s)$ which are either mass subcritical, mass critical or mass supercritical. Among its main ingredients is the study of the asymptotic behaviors of the positive solutions as $\lambda\rightarrow 0^+$ or $\lambda\rightarrow +\infty$ and the existence of an unbounded continuum of solutions in $(0, + \infty) \times H^1(\mathbb{R}^N)$.
\vskip0.23in

\noindent{\it   {\bf Key  words:}}   Global branch, Schr\"odinger equation, Positive normalized solution.

\vskip0.1in
\noindent{\it  {\bf 2010 Mathematics Subject Classification:}} 35A16, 35B40, 35A02, 35J60, 46N50.
\end{abstract}

\maketitle

\section{introduction}\label{sec1}
The aim of this paper is to study the existence, non-existence and multiplicity of positive  solutions with a prescribed mass for
nonlinear Schr\"odinger equations with general nonlinearities
\begin{equation}\label{eq:ENLS}
-i\frac{\partial}{\partial t}\Phi=\Delta \Phi+f(\mid\Phi\mid^2)\Phi \quad \hbox{in}\;\RN.
\end{equation}
The ansatz $\Phi(x,t)=e^{i\lambda t}u(x)$ for standing waves solutions leads to the  elliptic equation
\begin{equation}\label{eq:NLS}
-\Delta u+\lambda u=g(u)\quad \hbox{in}\;\RN,
\end{equation}
with $g(u)=f(\mid u\mid^2)u$. For physical reason we shall focus on solutions to (\ref{eq:NLS}) in $H^1(\RN)$.  In the past decades, this equation has been investigated by many authors, mainly in the fixed frequency case where $\lambda \in [0, + \infty)$ is prescribed.

An important feature of \eqref{eq:ENLS} is the conservation of mass: the $L^2$-norm $\|\Phi(\cdot,t)\|_2$ of a solution is
independent of $t\in\mathbb{R}$. In view of this property it is relevant to search for solutions satisfying a normalization constraint
\begin{equation}\label{eq:norm}
  \int_{\RN}\mid u\mid^2 dx =a, \quad a >0.
\end{equation}
When $g(s)=\mid s\mid^{p-1}s$ with $p+1\in (2, 2^*)$, where $2^*=+\infty$ for $N=1,2$ and $2^*=\frac{2N}{N-2}$ for $N\geq 3$ is the Sobolev critical exponent, the issue of positive normalized solutions to \eqref{eq:NLS}-\eqref{eq:norm} can be completely solved by scaling. Indeed, let $U_p$ be the unique positive radial solution to
\begin{equation*}
  -\Delta u+u=u^{p}\,\, \mbox{in}\, \, \RN;\quad u(x)\rightarrow 0\ \text{ as $\mid x\mid\rightarrow\infty$;}
\end{equation*}
cf.\ \cite{Kwong1989}.
Setting
$$U_{\lambda,p}(x):=\lambda^{\frac{1}{p-1}}U_p(\sqrt{\lambda}x),$$
one can check that, up to a translation, $U_{\lambda,p}$ is the unique positive solution to
$$ -\Delta u+\lambda u=u^{p}\, \, \mbox{in}\, \, \RN;\quad u(x)\rightarrow 0\ \text{ as $\mid x\mid\rightarrow\infty$.}$$
A direct computation shows that
\begin{equation*}
\|U_{\lambda,p}\|_{L^2(\RN)}^{2}=\lambda^{\frac{4-(p-1)N}{2(p-1)}}\|U_p\|_{L^2(\RN)}^{2}.
\end{equation*}
So one can see that if $p\neq \bar{p}:=1+\frac{4}{N}$, there exists a  unique $\displaystyle\lambda_a>0$ such that $\displaystyle \|U_{\lambda_a,p}\|_{L^2(\RN)}^{2}=a$.
That is, there exists a  positive normalized solution to \eqref{eq:NLS}-\eqref{eq:norm} for any $a>0$ whenever $p\neq 1+\frac{4}{N}$ (and it is unique up to a translation). While for the so-called mass critical exponent $p+1=2+\frac{4}{N}$, \eqref{eq:NLS}-\eqref{eq:norm} has a positive normalized solution if and only if $a=\|U_p\|_{L^2(\RN)}^{2}$ (with infinitely many solutions and $\lambda>0$).

However, if $g(s)$ is nonhomogeneous, such scaling argument is not valid anymore, more elaborate arguments are required.
The classical way to attack such problems is through a variational approach: one looks for critical points $u\in H^1(\RN)$ of the energy functional
\begin{equation*}
  J(u) = \frac{1}{2}\int_{\RN}\mid\nabla u\mid^2 \, dx- \int_{\RN}G(u) \, dx, \quad \mbox{where} \quad G(s):=\int_0^s g(t)dt,
\end{equation*}
subject to the constraint
$$ S_a:=\{u\in H^1(\RN): \|u\|_{L^2(\RN)}^{2}=a\}.$$
Note that in such an approach the parameter $\lambda \in \mathbb{R}$ appears as a Lagrange multiplier and is not a priori known. \medskip

One speaks of  a mass subcritical case if $J(u)$ is bounded from below on $S_a$ for any $a>0$, and of a  mass supercritical case if $J(u)$ is unbounded from below for any $a>0$. One also refers to a mass critical case when the boundedness from below depends on the value of $a>0$. Actually the case we are in crucially depends on the nonlinearity $g(s)$. So, when using a variational approach, in each case specific global conditions need to be imposed on $g(s)$ to treat the problem.

The study of the mass subcritical case, which can be traced back to the work of Stuart  \cite{Stuart1981,Stuart1982}, had seen a major advance with the introduction of the Compactness by Concentration approach of Lions \cite{Lions1984a,Lions1984b}. Nowadays, this case is rather well-understood and we refer to
\cite{Shibata2014, JeanLu2019, Jeanjean-Lu-2021, Stefanov19} for recent contributions.

In the mass supercritical case, $J(u)$ is unbounded from below on $S_a$ and thus there is no global minimizer. In \cite{Jeanjean1997}, the first author imposed some global conditions on $g(s)$ to guarantee the mountain pass geometry of $J(u)$ on $S_a$.
 Then, after solving some compactness issues, a normalized solution was obtained. Since \cite{Jeanjean1997} there has been numerous contributions in the mass supercritical, and mass critical cases. Let us just mention here \cite{BartschValeriola2013, BartschSoave2017, {BartschSoave2018}, BieganowskiMederski, JeanLu2020a, ikomatanaka2019, Soave-1}. We also refer to \cite{JeanjeanLe2020, Soave2020, Wei-Wu2021} for the treatment of nonlinearities which are Sobolev critical.

In the current paper, our aim is to present a different approach, based on the fixed point index and continuation arguments, to study normalized solutions problems. This approach does not rely on the variational structure, and thus we can treat in a unified way nonlinearities which are either mass subcritical, mass critical or mass supercritical. Apart for providing another framework  for the study of normalized problems, we believe it is already of interest for the existence of global branches of solutions. In contrast to the classical bifurcation methods, which concerns a branch bifurcating from some eigenvalue, in the spirit of Rabinowitz's global bifurcation theorem \cite{Rabinowitz1971}, our argument offers the possibility to investigate problems without eigenvalues.

It seems that Bartsch, Zou and the third author \cite{Bartsch-Zhong-Zou-2021} are the first to apply a global branch approach to study the existence of normalized solutions. They focused on the following coupled {S}chr\"{o}dinger system
\beq\lab{eq:system}
\begin{cases}
  -\De u+\lambda_1 u=\mu_1 u^3+\beta uv^2,~\hbox{in}~\RN,\\
  -\De v+\lambda_2 v=\mu_2 v^3+\beta vu^2,~\hbox{in}\;\RN,\\
  \int_{\RN}u^2=a^2\quad\hbox{and}\;\int_{\RN}v^2=b^2.
\end{cases}\
\eeq
The existence of normalized solutions for a large range of $\beta$ was obtained for arbitrary masses $a>0$ and $b>0$. The results of \cite{Bartsch-Zhong-Zou-2021} significantly extended the ones of Bartsch et al.\cite{BartschJeanjeanSoave16}, which rely on a variational approach on the $L^2$-spheres. However, we note that in \cite{Bartsch-Zhong-Zou-2021},  there is a positive branch bifurcating from the branch of semi-trivial solutions when $\beta$ is large. Also, the implicit theorem can be applied for $\beta=0$, since the problem \eqref{eq:system} is reduced to two independent scalar equations whose positive solutions are well identified  by \cite{Kwong1989}. In our scalar problem, we do not benefit from these features.  The global branch approach we shall develop is rather distinct from the one of \cite{Bartsch-Zhong-Zou-2021}. \medskip

The following assumptions will be used in the paper.
\begin{itemize}
\item[(G1)] $g \in C^1([0, + \infty))$, $g(s)>0$ for $s> 0$.
\item[(G2)] There exists some $(\alpha,\beta)\in \mathbb{R}_+^2$ satisfying
$$\begin{cases}
2<\alpha,\beta<2^*:=\frac{2N}{N-2}\quad&\hbox{if}~N\geq 3,\\
2<\alpha, \beta<2^*:=+\infty \quad&\hbox{if}~N=1,2,
\end{cases}$$
such that
$$\lim_{s\rightarrow 0^+}\frac{g'(s)}{s^{\alpha-2}}=\mu_1(\alpha-1)>0 \quad \mbox{and} \quad \lim_{s\rightarrow +\infty}\frac{g'(s)}{s^{\beta-2}}=\mu_2(\beta-1)>0.$$
\item[(G3)]  $-\Delta u=g(u)$ has no positive radial decreasing classical solution in $\mathbb{R}^N$.
\end{itemize}

\brr\lab{remark:20211208-r1}
\begin{itemize}
\item [(i)]
Note that (G1)-(G2) imply that $g(0)=0$.
Since we are interested in non-negative solutions to \eqref{eq:NLS}, it is not restrictive to assume, throughout the paper, that $g$ is extended to a continuous function on all $\mathbb{R}$ by setting $g(s) =0$, for all $s \in (-\infty,0].$ Then, by the weak maximum principle, any solution to \eqref{eq:NLS} is non-negative. The strong maximum principle also implies that it is strictly positive.
\item[(ii)]In Theorem \ref{thm:20210903-th2} we provide sufficient conditions on $g(s)$ under which condition (G3) holds. Actually, under (G1) and (G2), the condition (G3) holds, if $N=1,2$ or assuming that $\alpha-1 \leq \frac{N}{N-2}$ if $N \geq 3$. Furthermore, in the case of $N\geq 3$ with $\alpha-1>\frac{N}{N-2}$, if we suppose further that $sg(s)\leq 2^*G(s)$ in $[0,+\infty)$, then (G3) also holds.
\end{itemize}
\er\smallskip

In our setting, what is essentially imposed is the behavior of $g(s)$ at $0$ and $\infty$. Apart, to some extend, with condition (G3), we do not need  global conditions on $g$. Our main result is the following theorem.
\bt\lab{thm:20210822-th1}
Let $N\geq 1$ and assume that (G1)-(G3) hold. Then we have the following conclusions.
\begin{itemize}
\item[(i)]{\bf(mass subcritical case)} If $2<\alpha, \beta<2+\frac{4}{N}$,
    then for any given $a>0$, \eqref{eq:NLS}-\eqref{eq:norm} possesses a positive normalized solution $(\lambda, u_\lambda) \in (0, + \infty) \times H_{rad}^{1}(\RN) $.
\item[(ii)]{\bf (exactly mass critical case)} If $\alpha=\beta=2+\frac{4}{N}$, denote $\bar{\mu}_1:=\min\{\mu_1,\mu_2\}$ and $\bar{\mu}_2:=\max\{\mu_1,\mu_2\}$. Then  \eqref{eq:NLS}-\eqref{eq:norm} possesses at least one positive normalized solution   $(\lambda, u_\lambda) \in (0, + \infty) \times H_{rad}^{1}(\RN) $ provided $a\in \displaystyle\left(\bar{\mu}_{2}^{-\frac{N}{2}}\|U_{1+\frac{4}{N}}\|_2^2,\, \bar{\mu}_{1}^{-\frac{N}{2}}\|U_{1+\frac{4}{N}}\|_2^2\right)$. In particular, there exists some $0<a_1<\bar{\mu}_{2}^{-\frac{N}{2}}\|U_{1+\frac{4}{N}}\|_2^2$ and $a_2>\bar{\mu}_{1}^{-\frac{N}{2}}\|U_{1+\frac{4}{N}}\|_2^2$ such that  \eqref{eq:NLS}-\eqref{eq:norm} has no positive normalized solution for $a\in (0,a_1)\cup (a_2,+\infty)$.
\item[(iii)] {\bf(at most mass critical case)}
\begin{itemize}
\item[(iii-1)] If $2<\alpha<\beta=2+\frac{4}{N}$,  \eqref{eq:NLS}-\eqref{eq:norm} has at least one positive normalized solution $(\lambda, u_\lambda) \in (0, + \infty) \times H_{rad}^{1}(\RN) $ if $0<a<\mu_{2}^{-\frac{N}{2}}\|U_{1+\frac{4}{N}}\|_2^2$. Furthermore, there exists some $a^*\geq \mu_{2}^{-\frac{N}{2}}\|U_{1+\frac{4}{N}}\|_2^2$ such that \eqref{eq:NLS}-\eqref{eq:norm} has no positive normalized solution provided $a>a^*$.
\item[(iii-2)] If $2<\beta<\alpha=2+\frac{4}{N}$,  \eqref{eq:NLS}-\eqref{eq:norm} has at least one positive normalized solution $(\lambda, u_\lambda) \in (0, + \infty) \times H_{rad}^{1}(\RN) $
 if $a>\mu_{1}^{-\frac{N}{2}}\|U_{1+\frac{4}{N}}\|_2^2$. Furthermore, there exists some positive $a^*\leq \mu_{1}^{-\frac{N}{2}}\|U_{1+\frac{4}{N}}\|_2^2$ such that \eqref{eq:NLS}-\eqref{eq:norm} has no positive normalized solution provided $0<a<a^*$.
\end{itemize}
\item[(iv)] {\bf(mixed case)}
\begin{itemize}
\item[(iv-1)] If $2<\alpha< 2+\frac{4}{N}<\beta<2^*$, there exists some $0<a^*<M<+\infty$ such that for any $a\in (0,a^*)$, \eqref{eq:NLS}-\eqref{eq:norm} has at least two distinct positive normalized solutions $(\lambda_i, u_{\lambda_i}) \in (0, + \infty) \times H_{rad}^{1}(\RN),i=1,2$, while \eqref{eq:NLS}-\eqref{eq:norm} has no positive normalized solution provided $a>M$.
\item[(iv-2)] If $2<\beta< 2+\frac{4}{N}<\alpha<2^*$, there exists some $0<m<a^*<+\infty$ such that for any $a\in (a^*,+\infty)$, \eqref{eq:NLS}-\eqref{eq:norm} has at least two distinct positive normalized solutions $(\lambda_i, u_{\lambda_i}) \in (0, + \infty) \times H_{rad}^{1}(\RN),i=1,2$, while \eqref{eq:NLS}-\eqref{eq:norm} has no positive normalized solution provided $0<a<m$.
\end{itemize}
\item[(v)] {\bf(at least mass critical case)}
\begin{itemize}
\item[(v-1)] If $2+\frac{4}{N}=\alpha< \beta<2^*$, \eqref{eq:NLS}-\eqref{eq:norm} has at least one positive normalized solution $(\lambda, u_\lambda) \in (0, + \infty) \times H_{rad}^{1}(\RN) $ if $0<a<\mu_{1}^{-\frac{N}{2}}\|U_{1+\frac{4}{N}}\|_2^2$. Furthermore, there exists some $a^*\geq \mu_{1}^{-\frac{N}{2}}\|U_{1+\frac{4}{N}}\|_2^2$ such that \eqref{eq:NLS}-\eqref{eq:norm} has no positive normalized solution provided $a>a^*$.
\item[(v-2)] If $2+\frac{4}{N}=\beta< \alpha<2^*$,\eqref{eq:NLS}-\eqref{eq:norm} has at least one positive normalized solution $(\lambda, u_\lambda) \in (0, + \infty) \times H_{rad}^{1}(\RN) $ if $a>\mu_{2}^{-\frac{N}{2}}\|U_{1+\frac{4}{N}}\|_2^2$. Furthermore, there exists some positive $a^*\leq \mu_{2}^{-\frac{N}{2}}\|U_{1+\frac{4}{N}}\|_2^2$ such that \eqref{eq:NLS}-\eqref{eq:norm} has no positive normalized solution provided $0<a<a^*$.
\end{itemize}
\item[(vi)] {\bf(mass supercritical case)} If $2+\frac{4}{N}<\alpha, \beta<2^*$, then for any given $a>0$, \eqref{eq:NLS}-\eqref{eq:norm} possesses a normalized positive solution $(\lambda, u_\lambda) \in (0, + \infty) \times H_{rad}^{1}(\RN)$.
\end{itemize}
\et

There are two main steps in the proof of  Theorem \ref{thm:20210822-th1}. One is to show the existence of a continuum of positive solutions to \eqref{eq:NLS}. The other one is to understand the behavior of the $L^2$-norms of the positive solutions as $\lambda\rightarrow 0^+$ or $\lambda\rightarrow +\infty$. In that direction we have the following result which has its own interest.
\smallskip

In Theorem \ref{cro:20210824-xc1} below, and throughout the paper $C_{r,0}(\RN)$, denotes the space of continuous radial functions vanishing at $\infty$.
\bt\lab{cro:20210824-xc1}
Let $N \geq 1$ and assume that (G1)-(G3) hold. Let $U \in C_{r,0}(\RN)$ be the unique positive solution to
\beq\lab{eq:20210820-e6L}
-\Delta U+U=\mu_1 U^{\alpha-1}~\hbox{in}~\RN,
\eeq
and $V \in C_{r,0}(\RN)$ be the unique positive solution to
\beq\lab{eq:20210820-e6l}
-\Delta V+V=\mu_2 V^{\beta-1}~\hbox{in}~\RN.
\eeq
Then it holds,
\begin{itemize}
\item[(i)] Let $\{u_n\}_{n=1}^{\infty}\subset H^{1}(\RN)$ be positive solutions to \eqref{eq:NLS} with $\lambda=\lambda_n\rightarrow 0^+$. Then
    \begin{equation*}
		\lim_{n \rightarrow \infty }\| u_n\|_{\infty}=0, \quad  \quad
   \lim_{n \rightarrow \infty }\|\nabla u_n\|_2=0
   \end{equation*}
    and
   \begin{equation*}
   \lim_{n \rightarrow \infty }\|u_n\|_2=\begin{cases}
   0\quad &\alpha<2+\frac{4}{N},\\
    \|U\|_2& \alpha=2+\frac{4}{N},\\
    +\infty& \alpha>2+\frac{4}{N}.
    \end{cases}
    \end{equation*}
\item[(ii)] Let $\{u_n\}_{n=1}^{\infty}\subset H^{1}(\RN)$ be positive solutions to \eqref{eq:NLS} with $\lambda=\lambda_n\rightarrow + \infty$. Then
    \begin{equation*}
    \lim_{n \rightarrow \infty }\|u_n\|_\infty=+\infty,\lim_{n \rightarrow \infty }\|\nabla u_n\|_2=+\infty
    \end{equation*}
    and
    \begin{equation*}
    \lim_{n \rightarrow \infty }\|u_n\|_2=\begin{cases}
    +\infty\quad &\beta<2+\frac{4}{N},\\
    \|V\|_2& \beta=2+\frac{4}{N},\\
    0& \beta>2+\frac{4}{N}.
    \end{cases}
    \end{equation*}
\end{itemize}
\et

Let us now present the main steps of the proofs of Theorems \ref{thm:20210822-th1} and \ref{cro:20210824-xc1} and more globally the structure of the paper. \smallskip

In Section \ref{sec:prelim}, we collect a few facts about \eqref{eq:NLS}.  In particular, under the assumptions (G1)-(G2), we prove the radial symmetry and the monotonicity of the positive solutions associated to solutions of \eqref{eq:NLS} when $\lambda >0$. Also, through the derivation of some non-existence results, we reach the conclusion that positive solutions to \eqref{eq:NLS} should be searched assuming that $\lambda >0$.

In Section \ref{sec:compactness},
we derive a priori bounds on the $L^\infty$-norm of positive solutions depending on the parameter $\lambda >0$. Furthermore, we  prove the compactness of the set of positive solutions, both in the sense of $C_{r,0}(\RN)$ and in the sense of $H^1(\RN)$, when $\lambda >0$ lies in a compact interval away from $0$.

In Section \ref{sec:behavior}, we apply blow up arguments to explicit the behavior of positive solutions both in $C_{r,0}(\RN)$ and in $H^1(\RN)$, as $\lambda\rightarrow 0^+$ or $\lambda\rightarrow +\infty$. The proof of Theorem \ref{cro:20210824-xc1} is given there.

In Section \ref{sec:loc-uniqueness}, we investigate the uniqueness of positive solutions and prove in Theorem \ref{thm:loc-uniqueness} that under (G1)-(G3), \eqref{eq:NLS} possesses at most one positive solution for $\lambda>0$ small. It is also the case for $\lambda >0$ large just under (G1)-(G2).

In Section \ref{sec:Berestycki-Lions}, relying on classical results from \cite{BGK1983,BerestyckiLions1983}, we prove that \eqref{eq:NLS} has a positive solution  for any $\lambda >0$. Thus, in view of Theorem \ref{thm:loc-uniqueness}, equation \eqref{eq:NLS} has a unique positive solution $u_{\lambda} \in H_{rad}^1(\RN)$ if $\lambda>0$ is small.
We also show that $u_{\lambda} \in H_{rad}^1(\RN)$ is of mp-type, see Lemma \ref{lemma:MPL}, a property that will be useful in Section \ref{sec:global}.
 When $N \geq 2$, this uniqueness combined with the mountain pass variational characterization of the solution permits to show that $\{(\lambda, u_\lambda)\}$ is a curve for $\lambda >0$ small, see Lemma \ref{cro:20210901-xc1}. A separate argument is required when $N=1$, but it directly leads to the conclusion that, the set consisting of all positive solutions $\{(\lambda, u_\lambda):\lambda>0\}$ is connected, see Theorem \ref{thm:20210901-th1}.

In Section \ref{sec:global}, we focus on the case $N \geq 2$. Denoting
$$\mathcal{S}=\{(\lambda, u)\in (0, + \infty) \times H_{rad}^{1}(\RN)  : (\lambda, u)\;\hbox{solves \eqref{eq:NLS}}, \, u>0 \},$$
we prove the existence of an unbounded component $\widetilde{\mathcal{S}} \subset \mathcal{S}$ containing the local branch whose existence is guaranteed by Lemma \ref{cro:20210901-xc1}. More precisely, letting  $P_1: (0, +\infty) \times  H_{rad}^{1}(\RN)  \rightarrow (0, \infty)$ be the projection onto the $\lambda$-component, we prove that $P_1(\widetilde{\mathcal{S}}) = (0,+ \infty)$.

This result is obtained by reformulating  \eqref{eq:NLS} into a fixed point problem. For  $\lambda>0$ we define the map $\mathbb{T}_\lambda : H_{rad}^{1}(\RN)  \mapsto
H_{rad}^{1}(\RN)$ by
$$\mathbb{T}_\lambda (u)=(-\Delta +\lambda)^{-1}g(u).$$
Then, a solution to \eqref{eq:NLS} satisfies the fixed point equation
\begin{equation*}
\mathbb{T}_\lambda(u)=u,
\end{equation*}
where $\mathbb{T}_\lambda$  is  a completely continuous operator, see Lemma \ref{lemma:20210903-zl1}.

A key point to prove that the local branch belongs to an unbounded component of $\mathcal{S}$ is to show that the unique solution $u_{\lambda}$ with $\lambda >0$ small, carries a non null topological degree. Actually, this degree is $-1$.
It is proved combining the fact that $u_{\lambda}$ is a mp-type critical point of $J_{\lambda}(u)$ defined in \eqref{eq:20210901-e5},  with the fact that the Hessian of $J_{\lambda}(u)$ at $u_{\lambda}$ admits a simple non-positive first eigenvalue, see Lemma \ref{prop:20210821-p1}.

Once we know that  the topological degree is nontrivial, following classical arguments which go back to Leray and Schauder, we are able to conclude that $\widetilde{\mathcal{S}} \subset
{\mathcal{S}}$ is unbounded.  In fact, the information on the degree of $u_{\lambda}$ plays a similar role to the presence of a bifurcation point having an odd multiplicity in the classical result of P.H. Rabinowitz \cite{Rabinowitz1971}, which guarantees the existence of a global branch emanating from a bifurcation point. Now, in view of the priori bounds derived in Corollary  \ref{cro:20210824-c1} and
Lemma \ref{lemma:no-bifurcation}, we conclude that  $P_1(\widetilde{\mathcal{S}})= (0, + \infty)$, see Lemma \ref{cro:20210822-c1}.

In Section \ref{sec:application}, Theorem \ref{thm:20210822-th1} is proved combining Theorem \ref{cro:20210824-xc1} with the information that $P_1(\widetilde{\mathcal{S}})= (0, + \infty)$. \medskip

Throughout the paper we use the notation $\|u\|_p$ for the $L^p$-norm. $H_{rad}^{1}(\RN)$ denotes the radial subspace of $H^1(\RN)$. By $\rightharpoonup$  we denote the weak convergence in $H^1(\RN)$. Capital letters $C, C_\varepsilon,C_\lambda$ stand for positive constants which may depend on some parameters and whose precise value may change from line to line.

\section{Preliminaries and non existence results}\label{sec:prelim}
\numberwithin{equation}{section}

Our first lemma gather some properties of $g$ that hold under (G1)-(G2). Its proof is by direct computations.

\bl\lab{20210822-wl1}
Let $N \geq 1$ and assume that  (G1)-(G2) hold,
\begin{itemize}
\item[(i)]$\displaystyle \lim_{s\rightarrow 0^+}\frac{g(s)}{s^{\alpha-1}}=\mu_1$ and
$\displaystyle \lim_{s\rightarrow +\infty}\frac{g(s)}{s^{\beta-1}}=\mu_2$. \\
\item[(ii)] There exists some $C>0$ such that $$g(s)\leq C(s^{\alpha-1}+s^{\beta-1}), \, s\in \mathbb{R}^+$$
and for any $M>0$, there exists some $C_{M,1}\geq C_{M,2}>0$ such that
    $$C_{M,2} s^{\alpha-1}\leq g(s)\leq C_{M,1} s^{\alpha-1}, \forall s\in [0,M].$$
\item[(iii)] There exists some $\tilde{C}>0$ such that $$G(s)\leq \tilde{C}(s^{\alpha}+s^{\beta}), \, s\in \mathbb{R}^+,$$
where $G(s):=\int_0^sg(t)dt$ and for any $M>0$, there exists some $\tilde{C}_{M,1}\geq \tilde{C}_{M,2}>0$ such that
    $$\tilde{C}_{M,2} s^{\alpha-1}\leq G(s)\leq \tilde{C}_{M,1} s^{\alpha-1}, \, \forall s\in [0,M].$$
		\item[(iv)] If $\alpha>\beta$, we have, for some $D>0$,
   $$g'(s)\leq D s^{\beta-2}, \, \, g(s)\leq  D s^{\beta-1}, \, \, G(s) \leq D  s^{\beta}, \,  s\in \mathbb{R}^+.$$
\end{itemize}
\el

\bl\lab{lemma:20210823-l1}
Let $N \geq 1$ and assume that  (G1)-(G2) hold. Let $u$ be a non-negative classical solution to \eqref{eq:NLS} with $\lambda >0$, which vanishes at infinity. Then, up to a translation,
$u$ is radially symmetric and decreases with respect to $\mid x\mid$.
\el
\bp
The conclusion follows by \cite[Theorem 3'']{GNN-1979} for $N=1$ and by \cite[Theorem 2]{GNN-1981} for $N\geq 2$.
\ep

\brr\lab{remark:20210823-r1}
Note that $u \in H^1(\mathbb{R}^N)$ is a non-negative solution to \eqref{eq:NLS} with $\lambda >0$, if and only if $u$ is a non-negative classical solution vanishing at infinity. Indeed, if $u \in H^1(\mathbb{R}^N)$ is a solution to \eqref{eq:NLS},  by a standard Schauder estimate and the Harnack inequality, we can see that
$u\in C^2$ and is vanishing at infinity. Conversely, it is readily seen that any non-negative classical solution to \eqref{eq:NLS} where $\lambda >0$ which vanishes at infinity has an exponential decay and thus belongs to $H^1(\mathbb{R}^N)$. In particular,  the conclusions of Lemma  \ref{lemma:20210823-l1} hold for the non-negative solutions to \eqref{eq:NLS} with $\lambda >0$ which belong to $H^1(\mathbb{R}^N)$.
\er

Let us now present some non-existence results of non-negative solutions.
\bt\lab{thm:20210903-th1}
Let $N \geq 1$ and assume that $g \in C(\mathbb{R}^+, \mathbb{R}^+)$, where $\mathbb{R}^+$ denotes $[0,+\infty)$. Then  \eqref{eq:NLS} has no nontrivial non-negative classical solution if $\lambda <0$.
\et
\bp
If $N=1$ the result follows directly from \cite[Theorem 5]{BerestyckiLions1983}. If $N =2$ the result follows from \cite[Theorem 2.8]{Armstrong-Sirakov-2011} and if $N\geq 3$ from \cite[Theorem 2.1]{Armstrong-Sirakov-2011}.
\ep

We also have a non-existence result when $\lambda=0$, which provides sufficient conditions for (G3) to hold.

\bt\lab{thm:20210903-th2}
Let $N \geq 1$ and $g \in C(\mathbb{R}^+, \mathbb{R}^+)$ be such that,
\begin{equation}\label{behavior-zero}
\liminf_{s \to 0^+} \frac{g(s)}{s^{\gamma -1}} >0, \quad \mbox{for some} \quad \gamma >2.
\end{equation}

\begin{itemize}
\item[(i)]If $N=1,2$ or assuming that $\gamma-1 \leq \frac{N}{N-2}$ if $N \geq 3$,  \eqref{eq:NLS}  has no nontrivial non-negative classical solution when  $\lambda=0$.
\item[(ii)]If $N \geq 3$ assuming that $\gamma-1 < \frac{N+2}{N-2}$ and $sg(s) \leq 2^*G(s)$ in $\mathbb{R}^+$, \eqref{eq:NLS}  has no  nontrivial non-negative radial  classical solution when $\lambda=0$.
\end{itemize}
\et

\bp
(i) Suppose that $u\geq 0$ is a classical solution to \eqref{eq:NLS} with $\lambda=0$. Then from \eqref{behavior-zero} and using the fact that $u$ belongs to $L^{\infty}(\RN)$ we can find a constant $C>0$ such that
$$-\Delta u\geq C \,  u^{\gamma-1} ~\hbox{in}~\RN.$$
At this point, under our assumption on $\gamma >2$, the conclusion follows by \cite[Theorem 8.4]{QuittnerSouplet2007}.\\
(ii) We first remark that if $u\not\equiv 0$ is a radial solution, then $u(r)$ is positive in $[0,+\infty)$ by the maximum principle and it
solves
\beq\lab{eq:20210922-e2}
\begin{cases}
-u''-\frac{N-1}{r}u'=g(u), \quad r>0,\\
u'(0)=0.
\end{cases}
\eeq
Also, $u$ decreases with respect to $r\in [0,+\infty)$. Indeed, $$-(u''(r)+\frac{N-1}{r} u'(r))=g(u(r))>0,$$  implies that $(r^{N-1}u'(r))'<0$. Hence, $r^{N-1}u'(r)$ is decreasing in $[0,+\infty)$. By $r^{N-1}u'(r)\mid_{r=0}=0$, we see that $r^{N-1}u'(r)<0$ for $r>0$. Hence, $u'(r)<0$ for $r>0$.
This implies that $\displaystyle \tau:=\lim_{r\rightarrow \infty}u(r)\geq 0$ exists and since
$$g(\tau)=\lim_{r\rightarrow \infty}-(u''(r)+\frac{N-1}{r} u'(r))=0,$$
we have that  $\tau=0$, i.e.,
\begin{equation*}
\lim_{r\rightarrow +\infty}u(r)=0.
\end{equation*}
The corresponding Pohozaev function is defined by
\beq\lab{eq:20210922-e3}
P(r):=r^N\left(\frac{u'(r)^2}{2}+G(u(r))+\frac{N-2}{2}\frac{u(r)u'(r)}{r}\right).
\eeq
It is of class $C^1[0,+\infty)$ and satisfies $P(0)=0$. A direct computation shows that
\begin{equation*}
P'(r)=r^{N-1}\left(NG(u(r))-\frac{N-2}{2}u(r)g(u(r))\right), \, \forall  r>0.
\end{equation*}
Hence, under our assumption  $sg(s) \leq 2^*G(s)$ in $\mathbb{R}^+$, we see that $P(r)$ is non-decreasing in $[0,+\infty)$.

Let us now prove that $r u'(r)+(N-2)u(r)\geq 0$. It holds for $r=0$. Suppose there exists some $r_0>0$ such that
$$T:=r_0 u'(r_0)+(N-2)u(r_0)<0.$$
By $-\Delta u \geq 0$, we obtain that
$$\Big[r u'(r)+(N-2)u(r)\Big]'\leq 0.$$
That is, $r u'(r)+(N-2)u(r)$ is non-increasing in $\mathbb{R}^+$. So
$$r u'(r)+(N-2)u(r)\leq T<0, \, \forall r>r_0.$$
For $s>t>r_0$, since $u$ is positive, we have
$$-u(t)\leq \int_t^s u'(\tau)d\tau\leq \int_t^s \frac{T}{\tau}d\tau =T \ln\frac{s}{t}.$$
Let $s\rightarrow +\infty$, since $T<0$, we obtain a contradiction. Hence,
$r u'(r)+(N-2)u(r)\geq 0$ for all $r\in [0,+\infty)$
 and thus
\beq\lab{eq:20210922-e6}
\mid u'(r)\mid \leq (N-2)\frac{u(r)}{r}, \quad r>0.
\eeq
Using the fact $u(r)$ is decreasing in $[0,+\infty)$, it follows from \eqref{eq:20210922-e2} and \eqref{eq:20210922-e6} that
\begin{align}\label{G}
G(u(r))=&(N-1)\int_r^\infty \frac{u'(t)^2}{t}dt-\frac{u'(r)^2}{2} \nonumber\\
\leq&(N-1)\int_r^\infty \frac{u'(t)^2}{t}dt\\
\leq& C \int_r^\infty \frac{u(t)^2}{t^3}dt
\leq C \, \frac{u(r)^2}{r^2}. \nonumber
\end{align}
Since $u \in L^{\infty}(\RN)$, under our assumption on $\gamma >2$ there exists a constant $C>0$ such that $G(u(r)) \geq C u(r)^{\gamma}$, for all $r >0$. Thus, using
\eqref{G}, we obtain that
\beq\lab{eq:20210922-e7}
u(r)\leq C_0 r^{-\frac{2}{\gamma-2}}, \;\hbox{for $r$ large enough}.
\eeq
Then by \eqref{eq:20210922-e6}, we have that
\beq\lab{eq:20210922-e8}
\mid u'(r)\mid\leq C_1 r^{-\frac{\gamma}{\gamma-2}}, \;\hbox{for $r >0$ large enough}.
\eeq
So from \eqref{eq:20210922-e3}, \eqref{G}-\eqref{eq:20210922-e8} and since $\gamma<\frac{2N}{N-2}$, we conclude that
\begin{equation*}
\lim_{r\rightarrow +\infty} P(r)=0.
\end{equation*}
Hence, the monotonicity of $P(r)$ implies that $P(r)\equiv 0$ and thus $P'(r)\equiv 0$ in $[0,+\infty)$.
Namely,
\beq\lab{eq:20210922-e10}
u(r)g(u(r))\equiv 2^* G(u(r)), \,  r\in [0,+\infty).
\eeq
By the Hospital's rule, \eqref{eq:20210922-e7} and \eqref{eq:20210922-e10} can lead to
\begin{align*}
0<b:=&\lim_{s\to 0^+}\frac{sg(s)}{s^{\gamma}}=\lim_{s\to 0^+}\frac{\gamma G(s)}{s^{\gamma}}
=\lim_{r\to\infty}\frac{\gamma G(u(r))}{(u(r))^{\gamma}}\\
=&\lim_{r\to \infty}\frac{\frac{\gamma}{2^*} g(u(r))u(r)}{(u(r))^{\gamma}}
=\lim_{s\to 0^+}\frac{\frac{\gamma}{2^*} g(s)s}{s^{\gamma}}
=\frac{\gamma}{2^*} b,
\end{align*}
which implies that $\gamma=2^*$, a contradiction.
\ep

In view of Theorems \ref{thm:20210903-th1} and \ref{thm:20210903-th2} for the search of positive radial solutions to \eqref{eq:NLS} in $H^1(\RN)$ we shall focus on the case where $\lambda >0$.\\

Finally, for future reference we state a Proposition and recall some classical results.
\bo\lab{lemma:20210820-kernel}
Let $N \geq 1$, $1 < p < 2^* -1$, $\mu >0$. Then there exists a unique positive radial solution  $U_p^{\mu} \in H^1(\RN)$ to
\beq\lab{eq:l1}
- \Delta u + u = \mu u^{p}, \quad u \in H^1(\RN).
\eeq
Moreover,  the linearization of \eqref{eq:l1} at $U_p^{\mu}$, $\phi \mapsto - \Delta \phi + \phi -  \mu p (U_p^{\mu})^{p-1} \phi$, has a null kernel in $H_{rad}^1(\RN)$.
\eo
\bp
These results follow from the corresponding results due to \cite{Kwong1989} on the unique positive solution $U_p \in H^1(\RN)$ to
\beq\lab{eq:l2}
- \Delta u + u =  u^{p}
\eeq
after having observed that $U_p$ is a positive solution to \eqref{eq:l2} if and only if $U_p^{\mu}:= \mu^{\frac{1}{1- p}} U_{p}$ is a positive solution to \eqref{eq:l1}.
\ep
\begin{definition}\label{def-mptype}(see \cite[Definition 1]{Hofer84}).
Let $U$ be a nonempty open subset of a Banach space $F$ and $\Phi \in C^1(U, \mathbb{R})$. Suppose $u_0$ is a critical point of $\Phi$  at some level $d \in \mathbb{R}$. We say that $u_0$ is of mountainpass-type (mp-type) if for all open neighborhoods $W \subset U$ of $u_0$ the topological space $W \cap \{u \in F : \Phi(u) < d\}$ is nonempty and not path-connected.
\end{definition}
Then we have, see \cite[Theorem 2]{Hofer84}.
\bt\label{Hofer}
We assume that the condition $(\Phi)$ hold:
\begin{itemize}
\item[$(\Phi)$] Let $F$ be a real Hilbert space and $\Phi \in C^2(F, \mathbb{R})$, and assume the gradient $\Phi'$ of $\Phi$ has the form identity-compact. Furthermore, suppose that for all critical points $u_0$ of $\Phi$ the first eigenvalue $\lambda_1$ of the linearisation $\Phi''(u_0)$ at $u_0$ is simple provided $\lambda_1 \leq 0$.\medskip
\end{itemize}
\noindent
Then, if $u_0 \in U$ is an isolated critical point of $\Phi$ of mp-type, the local degree at $u_0$ is $-1.$
\et

Finally, we recall the following result,  due to Leray-Schauder, see \cite[Theorem 4.3.4]{Ambrosetti-Arcoya-11}.
\bo\lab{Proposition-degree}
Assume that $X$ is a real Banach space, $\Omega$ is a bounded, open subset of $X$ and $\Phi : [a,b] \times \overline{\Omega} \mapsto X$ is given by $\Phi(\lambda,u) = u - T(\lambda,u)$ with $T$ a compact map. Suppose also that
$$\Phi(\lambda, u) = u - T(\lambda, u) \neq 0, \quad \forall (\lambda, u) \in [a,b] \times \partial \Omega.$$
If,
$$deg(I - T(a, \cdot), \Omega, 0) \neq 0,$$
then there exists a compact connected set
 $\mathcal{C} \in \Sigma$ such that
$$\mathcal{C} \cap (\{a\} \times \Sigma_a) \neq \emptyset  \quad \mbox{and} \quad \mathcal{C} \cap (\{b\} \times \Sigma_b) \neq \emptyset.$$
Here,
$$\Sigma = \{ (\lambda, u) \in [a,b] \times \overline{\Omega} : \Phi(\lambda,u) =0 \}$$
and $\Sigma_{\lambda}$ denotes the $\lambda$-slice of $\Sigma$, i.e.
$$\Sigma_{\lambda}= \{u \in \overline{\Omega} : (\lambda, u) \in \Sigma \}.$$
\eo

\section{A priori bounds and compactness results}\label{sec:compactness}
\numberwithin{equation}{section}

For $0<\Lambda_1\leq \Lambda_2<+\infty$, we define the set
\beq\lab{eq:20210824-e2}
\mathcal{U}_{\Lambda_1}^{\Lambda_2}:=\left\{u\in H^1_{rad}(\RN): \hbox{$u $ is a non-negative solution to \eqref{eq:NLS} with $\lambda\in [\Lambda_1,\Lambda_2]$}\right\}
\eeq
and recall that $C_{r,0}(\RN)$ denotes the space of continuous radial functions vanishing at infinity. In view of Remark \ref{remark:20210823-r1}, $\mathcal{U}_{\Lambda_1}^{\Lambda_2} \subset C_{r,0}(\RN)$.

\bl\lab{lemma:prior-estimate1}
Let $N \geq 1$ and assume that (G1)-(G2) hold.
\begin{itemize}
\item[(i)] For any $M>0$, there exists $ C_M>0$ such that, for any non-negative solution $u \in H^1(\mathbb{R}^N)$ to \eqref{eq:NLS} with $\lambda\in (0,M)$,
$$\max_{x\in \RN}u(x)\leq C_M.$$
\item[(ii)] Let $0 < \Lambda_1 \leq \Lambda_2 < + \infty$. Then the set $\mathcal{U}_{\Lambda_1}^{\Lambda_2}$ is compact in $C_{r,0}(\RN)$.
\end{itemize}
\el

\begin{proof}

(i)
We  proceed  by contradiction, assuming there exists a sequence
$(\lambda_n, u_n) \in (0, M) \times H^1(\mathbb{R}^N) $ where $u_n \in H^1(\RN)$ is solution to \eqref{eq:NLS} with $\lambda=\lambda_n$ and
\begin{align*}
\max_{x\in \RN}u_n(x)\rightarrow +\infty \quad \mbox{as} \, n\rightarrow +\infty.
\end{align*}
We follow a blow up procedure introduced by Gidas and Spruck \cite{GidasSpruck1981}.
Without loss of generality, we may assume that
$\displaystyle M_n:=u_n(0)=\max_{x\in\RN}u_n(x)$.
Now we perform a rescaling, setting $x=\frac{y}{M_{n}^{\frac{\beta-2}{2}}}$ and defining
$\tilde{u}_n:\RN\rightarrow \mathbb{R}$ by
$$\tilde{u}_n(y):=\frac{1}{M_n}u_n\left(\frac{y}{M_{n}^{\frac{\beta-2}{2}}}\right).$$
Then $\displaystyle \tilde{u}_n(0)=\max_{y\in \RN}\tilde{u}_n(y)=1$ and
\begin{equation*}
-\Delta \tilde{u}_n=\frac{g(M_n \tilde{u}_n) }{M_{n}^{\beta-1}}-\frac{\lambda_n}{M_{n}^{\beta-2}}\tilde{u}_n.
\end{equation*}
If $\alpha\leq \beta,$ since $ \|\tilde{u}_n\|_\infty=1$ and $M_n\rightarrow +\infty$, recalling Lemma \ref{20210822-wl1}-(ii), we see that
\begin{equation*}
\frac{g(M_n \tilde{u}_n) }{M_{n}^{\beta-1}}\leq C\left(M_{n}^{\alpha-\beta} \tilde{u}_{n}^{\alpha-1}+\tilde{u}_{n}^{\beta-1}\right)\leq C\left(\tilde{u}_{n}^{\alpha-1}+\tilde{u}_{n}^{\beta-1}\right)\in L^\infty(\RN).
\end{equation*}
If $\alpha>\beta$, we also have that
\begin{equation*}
\frac{g(M_n \tilde{u}_n) }{M_{n}^{\beta-1}}\leq  \frac{C(M_n \tilde{u}_n)^{\beta-1}}{M_{n}^{\beta-1}}=C \tilde{u}_{n}^{\beta-1}\in L^\infty(\RN).
\end{equation*}
We also remark that $\frac{\lambda_n}{M_{n}^{\beta-2}}\tilde{u}_n\in L^\infty(\RN)$.
So, applying standard elliptic estimates, and passing to a subsequence if necessary, we may assume that $\tilde{u}_n\rightarrow \tilde{u}$ in $C_{loc}^{2}(\RN)$,
where $\tilde{u}$ is a nontrivial and non-negative bounded  radial solution to
\begin{equation*}
-\Delta \tilde{u}=\mu_2 \tilde{u}^{\beta-1}~\hbox{in}~\RN.
\end{equation*}
At this point Theorem \ref{thm:20210903-th2} provides a contradiction.

(ii)
Note that a bounded set $\mathcal{A}\subset C_{r,0}(\RN)$ is pre-compact if and only if $\mathcal{A}$ is equi-continuous on bounded sets and decay uniformly at infinity.
Firstly,  by a standard regularity argument we can check that the set $\mathcal{U}_{\Lambda_1}^{\Lambda_2}$ is  bounded in $C^2(\RN)$.  To  prove the uniform decay we argue by contradiction and assume that there exists a $\varepsilon>0$, that can be assumed as arbitrarily small, and sequences $\{u_n\}\subset C_{r,0}(\RN)$ and $r_n\rightarrow +\infty$ such that $u_n(r_n)=\varepsilon$ and $u_n$ solves \eqref{eq:NLS} with $\lambda=\lambda_n$.
Put $\bar{u}_n(r):=u_n(r+r_n)$, then
\begin{equation*}
-\left(\bar{u}''_n+\frac{(N-1)}{r+r_n} \bar{u}'_n\right)=-\lambda_n \bar{u}_n +g(\bar{u}_n), \,  r>-r_n.
\end{equation*}
Passing to subsequences (still denoted by $\lambda_n$ and $\bar{u}_n$) and then taking the limits, we get that $\lambda_n\rightarrow \lambda^*>0$ and that $\{u_n\}$ converges to $\bar{u}$, a  nontrivial solution  of the following equation
\begin{equation}\lab{eq:20211204-e1}
-\bar{u}''=-\lambda^* \bar{u}+g(\bar{u})~\hbox{in}~\mathbb{R}
\end{equation}
 with $\bar{u}(0)=\varepsilon, \bar{u}\geq 0$ and bounded. By Lemma \ref{lemma:20210823-l1}, $\bar{u}_n(r)$ is decreasing in $[-r_n,+\infty)$ and thus $\bar{u}$ is bounded and decreasing in $\mathbb{R}$. Hence, $\bar{u}(r)$ has a limit $\bar{u}_+$ at $r=+\infty$ and a limit $\bar{u}_-$ at $r=-\infty$. In particular, $0\leq \bar{u}_+\leq \bar{u}(r)\leq \bar{u}_-<+\infty, \forall r\in \mathbb{R}$ and $\bar{u}_-\geq \bar{u}(0)=\varepsilon>0$.
Here $\bar{u}_\pm$ satisfies
$$-\lambda^* \bar{u}_\pm+g(\bar{u}_\pm)=0.$$
Since $\lambda >0$ is bounded away from $0$, we have that $\lambda^*>0$. Under the assumption (G2), taking $\varepsilon >0$ smaller if necessary, we can assume that $\frac{g(s)}{s}<\lambda^*$ for $s\in (0,\varepsilon]$, which implies that $\bar{u}_+=0$.
Put $f(s):=-\lambda^* s+g(s)$ and $F(s):=\int_0^s f(t)dt$.
 Noting that $\displaystyle \lim_{t\rightarrow +\infty}\bar{u}'(t)=0$ and $\displaystyle \lim_{t\rightarrow +\infty}F(\bar{u}(t))=0$,
we have that
\beq\lab{eq:20211204-be2}
\frac{1}{2}\bar{u}'(r)^2=\int_{r}^{+\infty} -\bar{u}''(t)\bar{u}'(t)dt
=\int_{r}^{+\infty}f(\bar{u}(t))\bar{u}'(t)dt
=-F(\bar{u}(r)), \forall r\in \mathbb{R}.
\eeq
Under the assumptions (G1)-(G2), by \cite[Theorem 5]{BerestyckiLions1983}, there exist a unique solution $v$ (up to a translation) to the following equation
\beq\lab{eq:20211204-e2}
-v''=f(v)~\hbox{in}~\mathbb{R}, v\in C^2(\mathbb{R}),\lim_{r \rightarrow \pm \infty}v(r)=0  ~\hbox{and $v(r_0)>0$ for some $r_0\in \mathbb{R}$}.
\eeq
Without loss of generality, we suppose that $v(0)=\max_{r\in \mathbb{R}}v(r)$, then
$$\begin{cases}
v(r)=v(-r);\\
v(r)>0,r\in \mathbb{R};\\
v(0)=\xi_0;\\
v'(r)<0,r>0,
\end{cases}$$
where $\xi_0>0$ is determined by
$$\xi_0:=\inf\left\{\xi>0:F(\xi)=0\right\},$$
see \cite[Theorem 5]{BerestyckiLions1983} again.
By our choice of $\varepsilon >0$, we see that $f(s)<0, s\in (0,\varepsilon]$, and thus $\varepsilon<\xi_0$.
So there exists some $r_0>0$ such that $v(r_0)=\varepsilon$. Now, we let $\tilde{v}(r):=v(r+r_0)$, then
\beq\lab{eq:20211204-e3}
-\tilde{v}''=f(\tilde{v}) ~\hbox{in}~\mathbb{R},~\tilde{v}(0)=\varepsilon.
\eeq
Furthermore, noting that $\displaystyle\lim_{r\rightarrow +\infty}\tilde{v}(r)=0$, applying a similar argument as that in \eqref{eq:20211204-be2}, we conclude that
\beq\lab{eq:20211204-e4}
\tilde{v}'(r)=\begin{cases}
-\sqrt{-2F(\tilde{v}(r))}, \forall r\geq -r_0,\\
\sqrt{-2F(\tilde{v}(r))}, \forall r< -r_0.
\end{cases}
\eeq
Hence, both $\bar{u}$ and $\tilde{v}$ solve
\beq\lab{eq:20211204-e5}
\begin{cases}
-u''(r)=f(u(r))~\hbox{in}~\mathbb{R},\\
u(0)=\varepsilon,\\
u'(0)=-\sqrt{-2F(\varepsilon)}.
\end{cases}
\eeq
By the uniqueness of solutions of initial value problem, we conclude $\bar{u}\equiv \tilde{v}$ in $\mathbb{R}$. Thus,
$$\bar{u}_-=\lim_{r\rightarrow -\infty}\bar{u}(r)=\lim_{r\rightarrow -\infty}\tilde{v}(r)=0,$$
a contradiction to $\bar{u}_{-} \geq\varepsilon>0$.
\end{proof}

\bc\lab{cro:20210824-c1}
Let $N \geq 1$ and assume that (G1)-(G2) hold. Let  $0 < \Lambda_1 \leq \Lambda_2 < + \infty$. Then the set $\mathcal{U}_{\Lambda_1}^{\Lambda_2}$  is compact in $H^1(\RN)$.
\ec
\begin{proof}
By Lemma \ref{lemma:prior-estimate1}-(ii), the set $\mathcal{U}_{\Lambda_1}^{\Lambda_2}$ is compact in $C_{r,0}(\RN)$. So by (G2), we can find some $R>0$ large such that
$$\frac{g(u(x))}{u(x)}<\frac{\Lambda_1}{2}, \, \forall \mid x\mid\geq R, \, \forall u\in \mathcal{U}_{\Lambda_1}^{\Lambda_2}.$$
Then it follows that
\begin{equation*}
-\Delta u+\frac{\Lambda_1}{2} u\leq 0, \, \forall \mid x\mid\geq R, \, \forall u\in \mathcal{U}_{\Lambda_1}^{\Lambda_2}
\end{equation*}
and combining with the compactness of $\mathcal{U}_{\Lambda_1}^{\Lambda_2}$  in $C_{r,0}(\RN)$, there exists some $C>0$ so that
$$u(x)\leq C e^{-\frac{\Lambda_1}{2}\mid x\mid}, \, \forall \mid x\mid\geq R, \, \forall u\in \mathcal{U}_{\Lambda_1}^{\Lambda_2}.$$
Recalling Lemma \ref{lemma:prior-estimate1}-(ii), we can modify $C >0$ such that
$$u(x)\leq C e^{-\frac{\Lambda_1}{2}\mid x\mid}~\hbox{in}~\RN, \forall u\in \mathcal{U}_{\Lambda_1}^{\Lambda_2}$$
from which the boundedness of $\mathcal{U}_{\Lambda_1}^{\Lambda_2}$ in $H^1(\RN)$ follows. Now, for any sequence $\{u_n\}\subset \mathcal{U}_{\Lambda_1}^{\Lambda_2}$, we may assume that $u_n\rightharpoonup u$ in $H^1(\RN)$ and $\lambda_n\rightarrow \lambda^*\in [\Lambda_1,\Lambda_2]$.  In particular, $u$ is a positive radial solution to
 \beq\lab{eq:20210819-ze2}
 -\Delta u+\lambda^* u=g(u) \quad \mbox{in}\, \, H^1(\RN).
 \eeq
 Noting that Lemma \ref{20210822-wl1}-(ii) implies that  $g(u)u\leq C(\mid u\mid^\alpha+\mid u\mid^\beta)$, by the uniform exponential decay of $\{u_n\}$ and the Lebesgue dominated convergence theorem, we observe that
 \beq\lab{eq:20210819-ze3}
 \int_{\RN}g(u_n)u_n \, dx\rightarrow \int_{\RN}g(u)u \, dx \quad \mbox{as} \, n\rightarrow +\infty.
 \eeq
Hence, by \eqref{eq:20210819-ze2}-\eqref{eq:20210819-ze3}, we obtain that
\beq\lab{eq:20210819-ze4}
\|\nabla u_n\|_2^2+\lambda_n\|u_n\|_2^2\rightarrow \|\nabla u\|_2^2+\lambda^*\|u\|_2^2.
\eeq
Recalling that $\lambda_n\rightarrow \lambda^*\geq \Lambda_1>0, u_n\rightharpoonup u$ in $H^1(\RN)$, \eqref{eq:20210819-ze4}  implies that
 $u_n\rightarrow u$ in $H^1(\RN)$. That is, $\mathcal{U}_{\Lambda_1}^{\Lambda_2}$ is compact in $H^1(\RN)$.
\end{proof}

For future reference we also observe

\bl\lab{lemma:no-bifurcation}
Let $N \geq 2$ and assume that (G1)-(G2) hold. For any $0<\Lambda_1\leq \Lambda_2<+\infty$, let $\mathcal{U}_{\Lambda_1}^{\Lambda_2}$ be given by \eqref{eq:20210824-e2}.  Then there exists some $\delta_{\Lambda_1,\Lambda_2}>0$ such that
$$\|u\|_{H^1(\RN)}^{2}\geq \delta_{\Lambda_1,\Lambda_2}, \, \forall u\in \mathcal{U}_{\Lambda_1}^{\Lambda_2}.$$
\el
\begin{proof}
Since any $u\in \mathcal{U}_{\Lambda_1}^{\Lambda_2}$, solves \eqref{eq:NLS} for some $\lambda\in [\Lambda_1,\Lambda_2]$, using  Lemma \ref{20210822-wl1}-(ii),
\begin{align*}
\min\{1,\Lambda_1\} \|u\|_{H^1(\RN)}^{2} \leq& \|\nabla u\|_2^2+\Lambda_1\|u\|_2^2
\leq \|\nabla u\|_2^2+\lambda\|u\|_2^2\\
\leq&\int_{\RN}g(u)u dx
\leq C \left(\|u\|_\alpha^\alpha+\|u\|_\beta^\beta\right)\\
\leq& C \left(\|u\|_{H^1(\RN)}^{\alpha}+\|u\|_{H^1(\RN)}^{\beta}\right),
\end{align*}
which implies the existence of $\delta_{\Lambda_1,\Lambda_2} >0$.
\end{proof}

\section{Asymptotic behaviors of positive solutions for $\lambda\rightarrow 0^+$ and $\lambda\rightarrow +\infty$}\label{sec:behavior}
\numberwithin{equation}{section}

\subsection{The case of $\lambda\rightarrow 0^+$}

\bl\lab{lemma:20210820-l1}
Let $N \geq 1$ and assume that (G1)-(G2) hold. Let $\{u_n\}_{n=1}^{\infty}\subset H^1(\RN)$ be positive solutions to \eqref{eq:NLS} with $\lambda=\lambda_n\rightarrow 0^+$. Then
$$\limsup_{n\rightarrow +\infty}\|u_n\|_\infty<+\infty.$$
\el
\begin{proof}
It follows by Lemma \ref{lemma:prior-estimate1}-(i).
\end{proof}

\bl\lab{lemma:20210820-l2}
Let $N \geq 1$ and assume that (G1)-(G2) hold. Let  $\{u_n\}_{n=1}^{\infty}\subset H^1(\RN)$ be positive solutions to \eqref{eq:NLS} with $\lambda=\lambda_n\rightarrow 0^+$. Then
\begin{equation*}
\liminf_{n\rightarrow +\infty}\frac{\|u_n\|_{\infty}^{\alpha-2}}{\lambda_n}>0.
\end{equation*}
Furthermore, if (G3) holds, then
\beq\lab{eq:20210820-e2}
\|u_n\|_\infty\rightarrow 0~\hbox{as}~n\rightarrow +\infty.
\eeq
\el
\begin{proof}
For any sequence $(\lambda_n,u_n)$ with $\lambda_n\rightarrow 0^+$ and
\beq\lab{eq:20210823-e1}
-\Delta u_n+\lambda_n u_n=g(u_n) ~\hbox{in}~\RN, \, 0<u_n\in H^{1}(\RN),
\eeq
thanks to Lemma \ref{lemma:20210820-l1} and the invariance of \eqref{eq:NLS} under translation, we may suppose that $u_n(0)=\|u_n\|_\infty\leq M, \forall n\in \mathbb{N}$.
Setting
$$\bar{u}_n(x):=\frac{1}{u_n(0)}u_n\left(\frac{x}{\sqrt{\lambda_n}}\right),$$
we have $\bar{u}_n(0)=\|\bar{u}_n\|_\infty=1$ and
\beq\lab{eq:20210823-e2}
-\Delta \bar{u}_n+\bar{u}_n=\frac{1}{\lambda_n u_n(0)}g(u_n(0)\bar{u}_n)~\hbox{in}~\RN.
\eeq
By standard regularity arguments, $\bar{u}_n$ is a classical solution to \eqref{eq:NLS}. Taking $x=0$, we obtain
\begin{align*}
1=&\bar{u}_n(0)\leq -\Delta \bar{u}_n (0)+\bar{u}_n(0)\quad \hbox{(by maximum principle)}\\
=&\frac{1}{\lambda_n u_n(0)}g(u_n(0))
\leq C_{M,1}\frac{u_n(0)^{\alpha-2}}{\lambda_n}\quad \hbox{(by Lemma \ref{20210822-wl1}-(ii))}.
\end{align*}
Hence,
$$\liminf_{n\rightarrow +\infty}\frac{\|u_n\|_{\infty}^{\alpha-2}}{\lambda_n}>\frac{1}{C_{M,1}}>0.$$
Suppose now that $\displaystyle\liminf_{n\rightarrow +\infty}u_n(0)>0$. Since  $g(u_n)-\lambda_nu_n\in L^\infty(\RN)$ in \eqref{eq:20210823-e1}, passing to a subsequence, we may assume that $u_n\rightarrow u$ in $C_{loc}^{2}(\RN)$ with $u(0)=\max_{x\in \RN}u(x)>0$ and
$u$  a non-negative bounded radial decreasing function which solves
$$-\Delta u=g(u), ~\hbox{in}~\RN.$$
Hence, by (G3), we obtain that $u\equiv 0$, a contradiction to $u(0)>0$.
\end{proof}

\bl\lab{lemma:20210820-l3}
Let $N \geq 1$ and assume that (G1)-(G3) hold. Let  $\{u_n\}_{n=1}^{\infty}\subset H^{1}(\RN)$ be positive solutions to \eqref{eq:NLS} with $\lambda=\lambda_n\rightarrow 0^+$. Then
\begin{equation*}
\limsup_{n\rightarrow +\infty}\frac{\|u_n\|_{\infty}^{\alpha-2}}{\lambda_n}<+\infty.
\end{equation*}
\el
\begin{proof}
We argue by contradiction and assume there exists a sequence $\{u_n\}_{n=1}^{\infty}\subset H^{1}(\RN)$ of positive solutions to \eqref{eq:NLS} with $\lambda=\lambda_n\rightarrow 0^+$ such that
$$\displaystyle \frac{\|u_n\|_{\infty}^{\alpha-2}}{\lambda_n}\rightarrow +\infty.$$
Without loss of generality, we suppose that $u_n(0)=\|u_n\|_\infty, \forall n\in \mathbb{N}$. Setting
$$\tilde{u}_n(x):=\frac{1}{\|u_n\|_\infty}u_n\left(\frac{x}{\|u_n\|_{\infty}^{\frac{\alpha-2}{2}}}\right),$$
we have $\tilde{u}_n(0)=\|\tilde{u}_n\|_\infty=1$ and
\beq\lab{eq:20210823-e3}
-\Delta \tilde{u}_n=\frac{g(\|u_n\|_{\infty}\tilde{u}_n)}{\|u_n\|_{\infty}^{\alpha-1}}-\frac{\lambda_n}{\|u_n\|_{\infty}^{\alpha-2}} \tilde{u}_n.
\eeq
In view of Lemmas \ref{20210822-wl1} and \ref{lemma:20210820-l1}, the right hand side of \eqref{eq:20210823-e3} is in $L^\infty(\RN)$. So, passing to a subsequence if necessary, we may assume that $\tilde{u}_n\rightarrow \tilde{u}$ in $C_{loc}^{2}(\RN)$. Using \eqref{eq:20210820-e2}, we obtain that $\tilde{u}$ is a non-negative bounded solution to
\begin{equation*}
-\Delta \tilde{u}=\mu_1 \tilde{u}^{\alpha-1}~\hbox{in}~\RN.
\end{equation*}
Here again, Theorem \ref{thm:20210903-th2} provides a contradiction to $\tilde{u}(0)=\|\tilde{u}\|_\infty=1$.
\end{proof}

\bl\lab{lemma:20210820-l4}
Let $N \geq 1$ and assume that (G1)-(G3) hold. Let  $\{u_n\}_{n=1}^{\infty}\subset H^1(\RN)$ be positive radial solutions to \eqref{eq:NLS} with $\lambda=\lambda_n\rightarrow 0^+$.
Define
\beq\lab{eq:20210820-e4}
\bar{u}_n(x):=\frac{1}{u_n(0)}u_n\left(\frac{x}{\sqrt{\lambda_n}}\right),
\eeq
then $\bar{u}_n(x)\rightarrow 0$ as $\mid x\mid\rightarrow +\infty$ uniformly in $n\in \mathbb{N}$.
\el
\begin{proof}
By Lemmas \ref{lemma:20210820-l2} and  \ref{lemma:20210820-l3}, up to a subsequence, we may assume that
$$\frac{\|u_n\|_{\infty}^{\alpha-2}}{\lambda_n}\rightarrow C_0>0.$$
Take $\varepsilon>0$ small such that $C_0 \mu_1 \varepsilon^{\alpha-1}<1$. We argue by contradiction and suppose there exists a sequence $r_n\rightarrow +\infty$ such that $\bar{u}_n(r_n)=\varepsilon$. By changing the origin to $r_n$ and passing to the limit, we obtain a nontrivial solution $\bar{u}$ of the following equation,
\begin{equation}\lab{eq:20231113-e1}
-\bar{u}''=-\bar{u}+C_0\mu_1\bar{u}^{\alpha-1}, \;r\in \mathbb{R},
\end{equation}
with $\bar{u}(0)=\varepsilon, \bar{u}\geq 0$ and bounded. By Lemma \ref{lemma:20210823-l1}, we obtain that $\bar{u}$ is decreasing on $\mathbb{R}$. Hence, $\bar{u}$ has a limit $\bar{u}_+$ at $r=+\infty$ and a limit $\bar{u}_-$ at $r=-\infty$. In particular, $\bar{u}_\pm$ solve
$$-\bar{u}_\pm+C_0\mu_1 \bar{u}_{\pm}^{\alpha-1}=0.$$
So by $\bar{u}_+<\varepsilon\leq \bar{u}_-$, we obtain that $\bar{u}_+=0$ and $\bar{u}_-=(C_0\mu_1)^{\frac{1}{2-\alpha}}$. Then, since from (\ref{eq:20231113-e1}), we have that $-\bar{u}''\leq 0$ on $\mathbb{R}$ necessarily $\bar{u}'(0) <0$ and using again that $-\bar{u}''\leq 0$ on $(- \infty, 0]$ we get a contradiction with the fact that $\bar{u}$ is bounded.
\end{proof}

\bt\lab{thm:20210820-th1}{\bf (The behavior in the sense of $C_{r,0}(\RN)$ as $\lambda\rightarrow 0^+$)}
Let $N \geq 1$ and assume that (G1)-(G3) hold. Let  $\{u_n\}_{n=1}^{\infty}$ be positive radial solutions to \eqref{eq:NLS} with $\lambda=\lambda_n\rightarrow 0^+$.  Defining
\beq\lab{eq:20210820-e5}
v_n(x):=\lambda_{n}^{\frac{1}{2-\alpha}}u_n\left(\frac{x}{\sqrt{\lambda_n}}\right),
\eeq
and passing to a subsequence if necessary we have that $v_n\rightarrow U$ in $C_{r,0}(\RN)$ as $n\rightarrow +\infty$, where $U$ is given by Theorem \ref{cro:20210824-xc1}.
\et

\begin{proof}
By Lemmas \ref{lemma:20210820-l2}  and \ref{lemma:20210820-l3}, we have
\beq\lab{eq:20210824-xe1}
0<\liminf_{n\rightarrow \infty}\frac{u_n(0)^{\alpha-2}}{\lambda_n}\leq \limsup_{n\rightarrow \infty}\frac{u_n(0)^{\alpha-2}}{\lambda_n}<+\infty.
\eeq
Noting that $v_n$ solves the equation
\beq\lab{eq:20210824-xe2}
-\Delta v_n+v_n=\frac{g(\lambda_{n}^{\frac{1}{\alpha-2}}v_n)}{\lambda_{n}^{\frac{\alpha-1}{\alpha-2}}}~\hbox{in}~\RN
\eeq
by standard regularity argument, it is easy to see that $\{v_n\}$ is equi-continuous on bounded sets. On the other hand, we remark that
$v_n(x)= u_n(0) \lambda_{n}^{\frac{-1}{\alpha-2}} \bar{u}_n(x)$,  where $\bar{u}_n(x)$ is given by \eqref{eq:20210820-e4}. So, by Lemma \ref{lemma:20210820-l4} and \eqref{eq:20210824-xe1}, we see that $\{v_n\}$ decay to $0$ uniformly at $\infty$. Hence, $\{v_n\}$ is pre-compact in $C_{r,0}(\RN)$. Passing to a subsequence if necessary, we may assume that $v_n\rightarrow U\in C_{r,0}(\RN)$  where $U\in C_{r,0}(\RN)$ solves \eqref{eq:20210820-e6L} with $U(0)=\max_{x\in \RN}U(x)>0$.
\end{proof}

In order to apply these results to the study of the existence of solutions with prescribed mass, we also need to explicit the behavior of these solutions in $L^2(\RN)$ as $\lambda\rightarrow 0^+$.

\bt\lab{thm:20210820-th2}{\bf (The behavior in the sense of $H^1(\RN)$ as $\lambda\rightarrow 0^+$)}
Let $N \geq 1$ and assume that (G1)-(G3) hold. Let $\{u_n\}, \{\lambda_n\},\{v_n\}$ and $U$ be given by Theorem \ref{thm:20210820-th1}. Then we also have
$$v_n\rightarrow U~\hbox{in}~H^1(\RN).$$
\et
\begin{proof}
We can rewrite \eqref{eq:20210824-xe2} as
\begin{equation*}
-\Delta v_n+\left(1-\frac{g(\lambda_{n}^{\frac{1}{\alpha-2}}v_n)}{\lambda_{n}^{\frac{\alpha-1}{\alpha-2}}v_n}\right)v_n=0~\hbox{in}~\RN.
\end{equation*}
By the proof of Theorem \ref{thm:20210820-th1}, we know that $\{v_n\}$ is pre-compact in $C_{r,0}(\RN)$. So, in view of Lemma \ref{20210822-wl1}, we can find some $R >0$ large enough such that
$$1-\frac{g(\lambda_{n}^{\frac{1}{\alpha-2}}v_n)}{\lambda_{n}^{\frac{\alpha-1}{\alpha-2}}v_n}>\frac{1}{2}, \, \forall \mid x\mid\geq R, \, \forall n\in \mathbb{N}.$$
Then we have
\begin{equation*}
-\Delta v_n(x)+\frac{1}{2}v_n(x)\leq 0~, \forall \mid x\mid\geq R, \, \forall n\in \mathbb{N}.
\end{equation*}
It follows by the comparison principle and since, see Theorem \ref{thm:20210820-th1},$v_n \to U$ in $C_{r,0}(\RN)$ that one can find some $C_1,C_2>0$ such that
\begin{equation*}
v_n(x)\leq C_1 e^{-C_2\mid x\mid}~\hbox{in}~\RN, \, \forall n\in \mathbb{N}.
\end{equation*}
On the other hand, by Lemma \ref{20210822-wl1}-(i),
\begin{align*}
\|v_n\|_{H^1(\RN)}^{2}=&\int_{\RN}\frac{g(\lambda_{n}^{\frac{1}{\alpha-2}}v_n)}{\lambda_{n}^{\frac{\alpha-1}{\alpha-2}}}v_n \, dx
=(\mu_1+o(1))\|v_n\|_\alpha^\alpha,
\end{align*}
which implies that $\{v_n\}$ is bounded in $H^1(\RN)$.
Without loss of generality, we assume that $v_n\rightharpoonup \tilde{U}$ in $H^1(\RN)$.
By the Lebesgue dominated convergence theorem, we see that
\begin{equation*}
v_n\rightarrow \tilde{U}~\hbox{in}~L^p(\RN), \quad \forall p\in [1,+\infty).
\end{equation*}
Furthermore, since $\{v_n\}$ is pre-compact in $C_{r,0}(\RN)$ and $\lambda_n\rightarrow 0^+$, by Lemma \ref{20210822-wl1}-(i), we have that
$$\frac{g(\lambda_{n}^{\frac{1}{\alpha-2}}v_n)}{\lambda_{n}^{\frac{\alpha-1}{\alpha-2}}}v_n
=(\mu_1+o(1))v_{n}^{\alpha}, \quad a.e. ~\hbox{in}~\RN.$$
So $\tilde{U}$ solves
$$-\Delta \tilde{U}+\tilde{U}=\mu_1 \tilde{U}^{\alpha-1}~\hbox{in}~\RN \quad \mbox{with} \, \lim_{\mid x\mid\rightarrow +\infty}\tilde{U}(x)=0.$$
Hence $\tilde{U}=U$ and thus
\begin{align*}
\|v_n\|_{H^1(\RN)}^{2}-\|U\|_{H^1(\RN)}^{2}
=&\int_{\RN}\frac{g(\lambda_{n}^{\frac{1}{\alpha-2}}v_n)}{\lambda_{n}^{\frac{\alpha-1}{\alpha-2}}}v_n \,dx
-\mu_1\int_{\RN}U^\alpha \, dx\\
=&\mu_1\left(\|v_n\|_\alpha^\alpha -\|U\|_\alpha^\alpha\right)+o(1)=o(1),
\end{align*}
which implies that $v_n\rightarrow U$ in $H^1(\RN)$.
\end{proof}

\subsection{The case of $\lambda\rightarrow +\infty$}
\bl\lab{lemma:20210820-l5}
Let $N \geq 1$ and assume that (G1)-(G2) hold.  Let $\{u_n\}_{n=1}^{\infty}\subset H^{1}(\RN)$ be positive solutions to \eqref{eq:NLS} with $\lambda=\lambda_n\rightarrow +\infty$. Then
\beq\lab{eq:20210820-e7}
\liminf_{n\rightarrow +\infty}\|u_n\|_\infty=+\infty,
\eeq
and
\beq\lab{eq:20210820-e8}
\liminf_{n\rightarrow +\infty}\frac{\|u_n\|_{\infty}^{\beta-2}}{\lambda_n}>0.
\eeq
\el

\begin{proof}
By regularity, for any fixed $n$, $u_n\in C^2(\RN)$ and we may suppose that $u_n(0)=\|u_n\|_{\infty}$. We let $\bar{u}_n$ be given as in the proof of Lemma \ref{lemma:20210820-l2} and thus \eqref{eq:20210823-e2} holds. Taking $x=0$, by Lemma \ref{20210822-wl1}, we obtain that
\begin{align*}
1=&\bar{u}_n(0)\leq -\Delta \bar{u}_n(0)+\bar{u}_n(0)
\leq \frac{1}{\lambda_n u_n(0)} g(u_n(0))\\
\leq & C \, \frac{u_n(0)^{\alpha-2}+u_n(0)^{\beta-2}}{\lambda_n}.
\end{align*}
Since $\lambda_n\rightarrow +\infty$ and $\min\{\alpha,\beta\}>2$, we obtain \eqref{eq:20210820-e7}.
Furthermore, if $\alpha\leq \beta$, \eqref{eq:20210820-e8} trivially holds. If $\alpha>\beta$, by Lemma \ref{20210822-wl1}-(ii), $g(s)\leq Cs^{\beta-1}, s\in \mathbb{R}^+$, and we also obtain \eqref{eq:20210820-e8}.
\end{proof}

\bl\lab{lemma:20210820-l6}
Let $N \geq 1$ and assume that (G1)-(G2) hold. Let $\{u_n\}_{n=1}^{\infty}\subset H^{1}(\RN)$ be positive solutions to \eqref{eq:NLS} with $\lambda=\lambda_n\rightarrow +\infty$. Then
$$
\limsup_{n\rightarrow +\infty}\frac{\|u_n\|_{\infty}^{\beta-2}}{\lambda_n}<+\infty.
$$
\el
\begin{proof}
As in the proof of Lemma \ref{lemma:20210820-l3}, we  argue by contradiction and assume there exists a sequence
$\{u_n\}_{n=1}^{\infty}\subset H^{1}(\RN)$ of positive solutions to \eqref{eq:NLS} with $\lambda=\lambda_n\rightarrow +\infty$ such that $$  \frac{\|u_n\|_{\infty}^{\beta-2}}{\lambda_n}\rightarrow +\infty.$$
 Without loss of generality, we suppose that $u_n(0)=\|u_n\|_\infty, \forall n\in \mathbb{N}$. Setting
$$\tilde{u}_n(x):=\frac{1}{\|u_n\|_\infty}u_n\left(\frac{x}{\|u_n\|_{\infty}^{\frac{\beta-2}{2}}}\right),$$
we have $\tilde{u}_n(0)=\|\tilde{u}_n\|_\infty=1$ and
\beq\lab{eq:20210824-xbe1}
-\Delta \tilde{u}_n=\frac{g(\|u_n\|_{\infty}\tilde{u}_n)}{\|u_n\|_{\infty}^{\beta-1}}-\frac{\lambda_n}{\|u_n\|_{\infty}^{\beta-2}} \tilde{u}_n.
\eeq
By Lemma \ref{20210822-wl1}, the right hand side of \eqref{eq:20210824-xbe1} is in $L^\infty(\RN)$. Passing to the limit, we may assume that $\tilde{u}_n\rightarrow \tilde{u}$ in $C_{loc}^{2}(\RN)$ where $\tilde{u}$ is a non negative bounded solution to
\beq\lab{eq:20210824-xbe2}
-\Delta \tilde{u}=\mu_2 \tilde{u}^{\beta-1}~\hbox{in}~\RN.
\eeq
Again here, Theorem \ref{thm:20210903-th2} contradicts $\tilde{u}(0)=1$.
\end{proof}

\bt\lab{thm:20210820-th3}{\bf (The behavior in the sense of $C_{r,0}(\RN)$ as $\lambda\rightarrow +\infty$)}
Let $N \geq 1$ and assume that (G1)-(G2) hold. Let $\{u_n\}_{n=1}^{\infty}\subset H^{1}(\RN)$ be positive radial solutions to \eqref{eq:NLS} with $\lambda=\lambda_n\rightarrow +\infty$.
 Defining
\beq\lab{eq:20210820-e10}
v_n(x):=\lambda_{n}^{\frac{1}{2-\beta}}u_n\left(\frac{x}{\sqrt{\lambda_n}}\right),
\eeq
and passing to a subsequence if necessary we have that $v_n\rightarrow V$ in $C_{r,0}(\RN)$ as $n\rightarrow +\infty$, where $V$ is given by Theorem \ref{cro:20210824-xc1}.
\et

\begin{proof}
By Lemmas \ref{lemma:20210820-l5} and \ref{lemma:20210820-l6}, we see that
$$
0<\liminf_{n\rightarrow \infty}\frac{u_n(0)^{\beta-2}}{\lambda_n}\leq \limsup_{n\rightarrow \infty}\frac{u_n(0)^{\beta-2}}{\lambda_n}<+\infty,
$$
which implies $\{v_n\}$ is uniformly bounded in $L^\infty(\RN)$.
We note that $v_n$ solves
\beq\lab{eq:20210824-xbe4}
-\Delta v_n=-v_n+\frac{g(\lambda_{n}^{\frac{1}{\beta-2}}v_n)}{\lambda_{n}^{\frac{\beta-1}{\beta-2}}}~\hbox{in}~\RN.
\eeq
If $\alpha\leq \beta$, by Lemma \ref{20210822-wl1} and $\lambda_n\rightarrow +\infty$, we have that
$$\frac{g(\lambda_{n}^{\frac{1}{\beta-2}}v_n)}{\lambda_{n}^{\frac{\beta-1}{\beta-2}}}
\leq C \left(\lambda_{n}^{\frac{\alpha-\beta}{\beta-2}} v_{n}^{\alpha-1}+v_{n}^{\beta-1}\right),$$
which implies the right hand side of \eqref{eq:20210824-xbe4} is in $L^\infty(\RN)$. If $\alpha>\beta$, still by Lemma \ref{20210822-wl1}, we have that
$$\frac{g(\lambda_{n}^{\frac{1}{\beta-2}}v_n)}{\lambda_{n}^{\frac{\beta-1}{\beta-2}}}
\leq Cv_{n}^{\beta-1},$$
the right hand side of \eqref{eq:20210824-xbe4} is also of $L^\infty(\RN)$.
Thus, in view of standard elliptic estimates, we may assume that $v_n\rightarrow V$ in $C_{loc}^{2}(\RN)$. Hence, $\{v_n\}$ is equi-continuous on bounded sets. Also, proceeding exactly as in the proof of Lemma \ref{lemma:20210820-l4} we may check that  $\{v_n\}$ decay to $0$ uniformly at $\infty$.
Hence, $\{v_n\}$ is pre-compact in $C_{r,0}(\RN)$ and $v_n\rightarrow V$ in $C_{r,0}(\RN)$ where $0<V\in C_{r,0}(\RN)$ solves \eqref{eq:20210820-e6l}.
\end{proof}

Now, adapting in a straightforward way the proof of Theorem \ref{thm:20210820-th2} we also have
\bt\lab{thm:20210820-th4}{\bf (The behavior in the sense of $H^1(\RN)$ as $\lambda\rightarrow +\infty$)}
Let $N \geq 1$ and assume that (G1)-(G2) hold. Let $\{u_n\}, \{\lambda_n\},\{v_n\}$ and $V$ be given by Theorem \ref{thm:20210820-th3}. Then we also have
$$v_n\rightarrow V~\hbox{in}~H^1(\RN).$$
\et

Now we can give the
\begin{proof}[Proof of Theorem \ref{cro:20210824-xc1}]
(i) In view of Lemma \ref{lemma:20210823-l1} and Remark \ref{remark:20210823-r1} we can assume without restriction that $u_n$ is a radial function. Let $v_n$ be defined by \eqref{eq:20210820-e5}. Thus, noting that
 $$
 \begin{cases}
 \|\nabla u_n\|_2^2=\lambda_n^{1+\frac{2}{\alpha-2}-\frac{N}{2}} \|\nabla v_n\|_2^2,\\
 \|u_n\|_2^2=\lambda_n^{\frac{2}{\alpha-2}-\frac{N}{2}}\|v_n\|_2^2
 \end{cases}$$
and recalling, see Theorem \ref{thm:20210820-th2}, that $v_n \rightarrow U$ in $H^1(\RN)$ as $\lambda_n \rightarrow 0^+ $,  the conclusion follows.

 (ii) Similarly, for $v_n$ defined by \eqref{eq:20210820-e10}, we have
 $$
 \begin{cases}
 \|\nabla u_n\|_2^2=\lambda_n^{1+\frac{2}{\beta-2}-\frac{N}{2}} \|\nabla v_n\|_2^2,\\
 \|u_n\|_2^2=\lambda^{\frac{2}{\beta-2}-\frac{N}{2}}\|v_n\|_2^2.
 \end{cases}$$
Since $v_n \to V$ in $H^1(\RN)$ as $\lambda_n \to + \infty,$ see  Theorem \ref{thm:20210820-th4}, the conclusion follows.
\end{proof}

\section{Local uniqueness of positive solutions}\label{sec:loc-uniqueness}
\numberwithin{equation}{section}

In this section, we prove the uniqueness of positive solutions to \eqref{eq:NLS} provided $\lambda >0$ small or large enough.
The uniqueness will help us to prove that the set consisting of $(\lambda, u_\lambda)$ with $\lambda>0$ small, is indeed a curve in Section \ref{sec:Berestycki-Lions}.
More precisely,

\bt\lab{thm:loc-uniqueness}
Let $N \geq 1$ and assume that (G1)-(G2) hold. Then \eqref{eq:NLS} has at most one positive solution for $\lambda >0$ large enough. If in addition (G3) holds, it is also the case for $\lambda >0$ small enough.
\et

\begin{proof}
(i) We first consider the case where $\lambda >0$ is small.
We argue by contradiction and suppose there exist two families of positive solutions $u_{\lambda}^{(1)}$ and $u_{\lambda}^{(2)}$ to \eqref{eq:NLS} with $\lambda\rightarrow 0^+$. Let
$$
  v_{\lambda}^{(i)}(x)
	 :=\lambda^{-\frac{1}{\alpha-2}}u_{\lambda}^{(i)}
	       \left(x/\sqrt{\lambda}\right),\quad i=1,2.
$$
Then
$v_{\lambda}^{(1)}(x),
 v_{\lambda}^{(2)}(x)\in H_{rad}^{1}(\RN)$
are two families of positive radial solutions to
$$-\Delta v+v=\frac{g(\lambda^{\frac{1}{\alpha-2}}v)}{\lambda^{\frac{\alpha-1}{\alpha-2}}}.$$
By Theorems \ref{thm:20210820-th1} and  \ref{thm:20210820-th2},
$$v_{\lambda}^{(i)}(x)\rightarrow U\;\hbox{as $\lambda\rightarrow 0^+$ both in}\;C_{r,0}(\RN)~\hbox{and in}~H^1(\RN), \;i=1,2.$$
We study the normalization
$$
\xi_\lambda:=\frac{v_{\lambda}^{(1)}-v_{\lambda}^{(2)}}{\left\|v_{\lambda}^{(1)}-v_{\lambda}^{(2)}\right\|_{\infty}}.
$$
By Lemma \ref{20210822-wl1} and the mean value theorem, for any $x \in \RN$, there exists some $\theta(x)\in [0,1]$ such that
\begin{align*}
&\frac{g(\lambda^{\frac{1}{\alpha-2}}v_{\lambda}^{(1)}(x))}{\lambda^{\frac{\alpha-1}{\alpha-2}}}-\frac{g(\lambda^{\frac{1}{\alpha-2}}v_{\lambda}^{(2)}(x))}{\lambda^{\frac{\alpha-1}{\alpha-2}}}\\
=&\lambda^{-\frac{\alpha-1}{\alpha-2}}g'\left(\lambda^{\frac{1}{\alpha-2}}[\theta(x)v_{\lambda}^{(1)}(x)+(1-\theta(x))v_{\lambda}^{(2)}(x)]\right) \lambda^{\frac{1}{\alpha-2}} \left(v_{\lambda}^{(1)}(x)-v_{\lambda}^{(2)}(x)\right)\\
=&\lambda^{-1}g'\left(\lambda^{\frac{1}{\alpha-2}}[\theta(x)v_{\lambda}^{(1)}(x)+(1-\theta(x))v_{\lambda}^{(2)}(x)]\right) \left(v_{\lambda}^{(1)}(x)-v_{\lambda}^{(2)}(x)\right)\\
=&\lambda^{-1}[1+o(1)]\mu_1(\alpha-1) \left(\lambda^{\frac{1}{\alpha-2}}[\theta(x)v_{\lambda}^{(1)}(x)+(1-\theta(x))v_{\lambda}^{(2)}(x)]\right)^{\alpha-2}  \left(v_{\lambda}^{(1)}(x)-v_{\lambda}^{(2)}(x)\right)\\
=&[1+o(1)]\mu_1(\alpha-1) \left(\theta(x)v_{\lambda}^{(1)}(x)+(1-\theta(x))v_{\lambda}^{(2)}(x)\right)^{\alpha-2}  \left(v_{\lambda}^{(1)}(x)-v_{\lambda}^{(2)}(x)\right).
\end{align*}
So we have that
\beq\lab{eq:20210820-xbue1}
-\Delta \xi_\lambda =-\xi_\lambda+[1+o(1)]\mu_1(\alpha-1) \left(\theta(x)v_{\lambda}^{(1)}(x)+(1-\theta(x))v_{\lambda}^{(2)}(x)\right)^{\alpha-2} \xi_\lambda.
\eeq
Recalling Lemma \ref{lemma:20210820-l4}, and the facts that $\|\xi_\lambda\|_\infty=1, \theta(x)\in [0,1]$ and $v_{\lambda}^{(i)}\rightarrow U$ in $C_{r,0}(\RN)$, one can see that the right hand side of \eqref{eq:20210820-xbue1} is in $L^\infty(\RN)$. Hence, passing to a  subsequence if necessary, we can assume that $\xi_\lambda\rightarrow \xi$ in $C_{loc}^{2}(\RN)$, where $\xi$ is a radial bounded function satisfying
$$
-\Delta \xi+\xi=(\alpha-1)\mu_1 U^{\alpha-2}\xi.
$$
Since $\|\xi\|_\infty=1$, standard elliptic estimates imply that $\xi$ is a strong solution. Then by the decay of $U$ and applying a comparison principle, we obtain that $\xi$ is exponentially decaying to $0$ as $\mid x\mid\to \infty$. Hence, $\xi\in C_{r,0}(\RN)\cap H_{rad}^1(\RN)$. At this point, Proposition \ref{lemma:20210820-kernel} provides a contradiction.

(ii) Now, we consider the case of $\lambda >0$ large.
We also argue by contradiction and suppose there exist two families of positive solutions $u_{\lambda}^{(1)}$ and $u_{\lambda}^{(2)}$ to \eqref{eq:NLS} with $\lambda\rightarrow +\infty$. Let
$$
  v_{\lambda}^{(i)}(x)
	 :=\lambda^{-\frac{1}{\beta-2}}u_{\lambda}^{(i)}
	       \left(x/\sqrt{\lambda}\right),\quad i=1,2.
$$
Then
$v_{\lambda}^{(1)}(x),
 v_{\lambda}^{(2)}(x)\in H_{rad}^{1}(\RN)$
are two families of positive solutions to the problem
$$-\Delta v+v=\frac{g(\lambda^{\frac{1}{\beta-2}}v)}{\lambda^{\frac{\beta-1}{\beta-2}}}.$$
By Theorems \ref{thm:20210820-th3} and  \ref{thm:20210820-th4},
$$v_{\lambda}^{(i)}(x)\rightarrow V\;\hbox{as $\lambda\rightarrow +\infty$ both in}\;C_{r,0}(\RN)~\hbox{and also in}~H^1(\RN), \;i=1,2.$$
We also consider  the normalization
$$
\xi_\lambda:=\frac{v_{\lambda}^{(1)}-v_{\lambda}^{(2)}}{\left\|v_{\lambda}^{(1)}-v_{\lambda}^{(2)}\right\|_{\infty}}.
$$
A direct computation shows that there exists some $\theta(x)\in [0,1]$ such that
\beq\lab{eq:20210820-xbue2}
-\Delta \xi_\lambda =-\xi_\lambda +\lambda^{-1}g'\left(\lambda^{\frac{1}{\beta-2}}[\theta(x)v_{\lambda}^{(1)}(x)+(1-\theta(x))v_{\lambda}^{(2)}(x)]\right) \xi_\lambda~\hbox{in}~\RN.
\eeq
By Lemma \ref{20210822-wl1}-(iv), if $\alpha>\beta$, we have
\begin{align*}
&\lambda^{-1}g'\left(\lambda^{\frac{1}{\beta-2}}[\theta(x)v_{\lambda}^{(1)}(x)+(1-\theta(x))v_{\lambda}^{(2)}(x)]\right) \\
\leq& C [\theta(x)v_{\lambda}^{(1)}(x)+(1-\theta(x))v_{\lambda}^{(2)}(x)]^{\beta-2},
\end{align*}
which implies the right hand side of \eqref{eq:20210820-xbue2} is in $L^\infty(\RN)$ by $\|\xi_\lambda\|_\infty=1, \theta(x)\in [0,1]$ and $v_{\lambda}^{(i)}\rightarrow V$ in $C_{r,0}(\RN)$.
On the other hand, if $\alpha\leq \beta$,
\begin{align*}
&\lambda^{-1}g'\left(\lambda^{\frac{1}{\beta-2}}[\theta(x)v_{\lambda}^{(1)}(x)+(1-\theta(x))v_{\lambda}^{(2)}(x)]\right) \\
\leq& C\left(\lambda^{\frac{\alpha-\beta}{\beta-2}}[\theta(x)v_{\lambda}^{(1)}(x)+(1-\theta(x))v_{\lambda}^{(2)}(x)]^{\alpha-2}
+[\theta(x)v_{\lambda}^{(1)}(x)+(1-\theta(x))v_{\lambda}^{(2)}(x)]^{\beta-2}\right).
\end{align*}
So by $\lambda\rightarrow +\infty$,
\begin{align*}
&\lambda^{-1}g'\left(\lambda^{\frac{1}{\beta-2}}[\theta(x)v_{\lambda}^{(1)}(x)+(1-\theta(x))v_{\lambda}^{(2)}(x)]\right)\\
\leq &C \left([\theta(x)v_{\lambda}^{(1)}(x)+(1-\theta(x))v_{\lambda}^{(2)}(x)]^{\alpha-2}
+[\theta(x)v_{\lambda}^{(1)}(x)+(1-\theta(x))v_{\lambda}^{(2)}(x)]^{\beta-2}\right).
\end{align*}
Using the facts that  $\|\xi_\lambda\|_\infty=1, \theta(x)\in [0,1]$ and $v_{\lambda}^{(i)}\rightarrow V$ in $C_{r,0}(\RN)$, we see the right hand side of \eqref{eq:20210820-xbue2} is in $L^\infty(\RN)$.
Hence, passing to a subsequence we can assume that $\xi_\lambda \rightarrow \xi$ in $C_{loc}^{2}(\RN)$, where $\xi$ is a radial bounded function satisfying
$$
-\Delta \xi+\xi=(\beta-1)\mu_2 V^{\beta-2}\xi.
$$
As in the case of $\lambda>0$ small, Proposition \ref{lemma:20210820-kernel} provides a contradiction.
\end{proof}

\section{Existence of a curve of positive solutions for $\lambda >0$ small.}\label{sec:Berestycki-Lions}
\numberwithin{equation}{section}
In this section we prove the existence of a curve of positive solutions to \eqref{eq:NLS} when $\lambda >0$ is small. In the particular case of $N=1$ we prove that this curve is global. We shall use the classical results recalled below.

\begin{proposition}\lab{pro:20210901-xc1}
Let $N \geq 1$ and assume that (G1)-(G2) hold. For any $\lambda >0$ there exists a positive radial solution $u_{\lambda} \in H^1(\RN)$ to \eqref{eq:NLS}  which, as any solution to \eqref{eq:NLS}, satisfies the Pohozaev identity
\beq\lab{eq:20210901-e11}
\frac{N-2}{2}\|\nabla u\|_2^2+\frac{N}{2}\lambda\|u\|_2^2=N\int_{\RN}G(u) \, dx.
\eeq
Moreover, defining the functional $J_\lambda: H^{1}(\RN)\rightarrow \mathbb{R}$ by
$$
J_\lambda(u):=\frac{1}{2}\left(\|\nabla u\|_2^2+\lambda \|u\|_2^2\right)-\int_{\RN}G(u) \, dx
$$ we have
\beq\lab{eq:20210901-e5}
J_\lambda(u_{\lambda}) = m_\lambda:=\inf_{\gamma\in \Gamma}\max_{t\in [0,1]}J_\lambda(\gamma(t)),
\eeq
where
\beq\lab{eq:20210901-e6}
\Gamma_{\lambda}:= \Big\{ \gamma \in C([0, 1], H^1(\RN)): \gamma(0)=0, J_{\lambda}(\gamma(1))<0 \Big\}.
\eeq
In particular, \beq\lab{eq:20210901-e12}
m_{\lambda}=\frac{1}{N}\|\nabla u_{\lambda}\|_2^2.
\eeq
\end{proposition}

\bp In \cite{BerestyckiLions1983} for $N\geq 3$ and in \cite{BGK1983} for $N=2$ the existence of a least action solution was established for \eqref{eq:NLS} under very general assumptions which are satisfied under (G1)-(G2) and assuming that $g$ is odd.   Later in \cite{JeanTanaka2002}, for $N\geq 2$ and in \cite{JeanTanaka2003} for $N=1$, the mountain pass characterization of this least action solution was established. Let us check that these results also hold for our nonlinearity $g(s)$ which is not odd. Clearly, the solution $u_{\lambda}$ obtained by replacing $g(s)$ by $\widetilde{g}(s)$ with
$$ \widetilde{g}(s) = \begin{cases}
g(s)\quad&\hbox{if} \, s \geq 0,\\
-g(-s) \quad&\hbox{if} \, s \leq 0
\end{cases}$$
being non-negative is also a least action solution of the functional $J_{\lambda}(u)$. Now, observe that since $g(s) \geq 0$ it holds that
$$\widetilde{G}(s) := \int_0^s \widetilde{g}(t) dt \geq G (s), \quad \forall s \in \R.$$
Thus, recalling, see \cite{JeanTanaka2002,JeanTanaka2003} that, for any positive least action solution, there exists a path $\gamma \in \Gamma_{\lambda}$  satisfying
$$\max_{t \in [0,1]} J_{\lambda}(\gamma(t)) = m_{\lambda}$$
such that $\gamma(t) (x) >0$ for all $x \in \RN$, $t \in (0,1]$, we deduce that the mountain pass characterization given in \eqref{eq:20210901-e5}-\eqref{eq:20210901-e6} is preserved even if $g$ is not odd.
Finally, we note that \eqref{eq:20210901-e12} follows directly combining \eqref{eq:20210901-e11} and \eqref{eq:20210901-e5}.
\ep

Recalling that the definition of a critical point of mp-type is given in Definition \ref{def-mptype}, we have
\begin{lemma}\label{lemma:MPL}
Let $N \geq 2$ and assume that (G1)-(G2) hold. Any solution $w \in H^1(\RN)$ to \eqref{eq:NLS} which satisfies $J_{\lambda}(w) = m_{\lambda}$ is of mp-type.
\end{lemma}

\begin{proof}
Let $w \in H^1(\RN)$ be a critical point of $J_{\lambda}$ with $J_{\lambda}(w) = m_{\lambda}$. We need to show that:
 for any open neighborhood $W \subset H^1(\mathbb{R}^N)$ of $w$, the set
$$ W_w^- := W \cap \{ u \in  H^1(\mathbb{R}^N) : J_{\mu}(u) < m_{\mu}\}$$
is nonempty and not path-connected. \\
Since $W$ is open, it contains a ball $B(w, 4r)$ where
$$B(w, 4r) = \{u \in H^1(\mathbb{R}^N) : \|u-w \|_{H^1(\mathbb{R}^N)} < 4r \}.$$
Using \cite[Lemma 4.1]{Jeanjean-Lu-2021} with $\delta = 2 r$ and $M >0$ arbitrary fixed, we deduce that there exists a constant $T >0$ and a
continuous path $\gamma : [0, T] \to H^1(\mathbb{R}^N)$ satisfying
    \begin{itemize}
      \item[$(i)$] $\gamma(0) = 0$, $J_{\lambda} (\gamma(T)) < -1$, $\max_{t \in [0, T]} J_{\lambda}(\gamma(t)) = J_{\lambda}(w)$; \smallskip
			
      \item[$(ii)$] $\gamma(\tau) = w$ for some $\tau \in (0, T)$, and
        \begin{equation*}
          J_{\lambda}(\gamma(t)) < J_{\lambda}(w)
        \end{equation*}
      for any $t \in [0, T]$ such that $\|\gamma(t) - w\|_{H^1(\mathbb{R}^N)} \geq 2r $.
    \end{itemize}
Note that, after a reparametrization, we can assume $T >0$ to be $1$.
		We fix $\tau_1 < \tau < \tau_2$ such that
		$$ \|\gamma(\tau_1) - w\|_{H^1(\mathbb{R}^N)} = \|\gamma(\tau_2) - w\|_{H^1(\mathbb{R}^N)} = 3r.$$
		The points $\gamma(\tau_1)$ and $\gamma(\tau_2)$ belong to $W_w^-$ and cannot be connected inside the set $W_w^-$.
		Indeed, suppose that they can be connected by $s: (\tau_1, \tau_2) \mapsto W_w^-$ with $s(\tau_1) = \gamma(\tau_1)$ and $s(\tau_2) = \gamma(\tau_2)$. Then, considering the path
		\begin{equation*}
    \tilde{\gamma}(t)  = \left\{
      \begin{aligned}
        &  \gamma(t), & &  \qquad \text{for} ~ t \in [0, \tau_1],\\
        &   s(t), & & \qquad \text{for} ~ t \in [\tau_1, \tau_2), \\
        &  \gamma(t), & & \qquad \text{for} ~ t \in [\tau_2, T]
      \end{aligned}
    \right.
  \end{equation*}
	we would have that $\tilde{\gamma} \in \Gamma_{\lambda}$  with
	$$\max_{t \in [0,T]}J_{\lambda}(\tilde{\gamma}(t)) < m_{\lambda}$$
	in contradiction with the definition of $m_{\lambda}.$
\end{proof}

\brr\lab{remark:20211130-xr1}
The property that any least action solution of the general equations considered in \cite{BerestyckiLions1983,BGK1983} is of mp-type can be proved exactly as in Lemma \ref{lemma:MPL}.
\er

\bl\lab{cro:20210901-xc1}
Let $N \geq 2$ and assume that (G1)-(G3) hold. There exists some $\lambda_0>0$ small, such that for any $\lambda\in (0,\lambda_0)$,  \eqref{eq:NLS} has a unique positive solution $u_\lambda \in H^1(\RN)$.  Furthermore,  the map $\lambda\mapsto u_\lambda, \lambda\in (0,\lambda_0)$ is continuous. That is, $\left\{(\lambda, u_\lambda):\lambda\in (0,\lambda_0)\right\}$ is a curve in $\mathbb{R}\times H_{rad}^{1}(\RN)$.
\el
\begin{proof}
Combining Proposition  \ref{pro:20210901-xc1} and Theorem \ref{thm:loc-uniqueness} we see that there exists a unique positive solution $u_{\lambda} \in H^1(\RN)$ to \eqref{eq:NLS} provided $\lambda >0$ small enough. For $N\geq 3$, just under (G1)-(G2), the result follows by
 \cite[Corollary 3.5]{Jeanjean-arxiv}, see also \cite[Lemma 19]{Shatah1985}. Here, assuming in addition that (G3) holds we give  a proof for any $N\geq 2$.
 Under (G1)-(G2) we can find some $s^*>0$ and some $\gamma>2$ such that
 \beq\lab{eq:202109020e1}
 g(s)s\geq \gamma G(s), \,  \forall s\in [0,s^*].
 \eeq
 When in addition (G3) holds, by Lemma \ref{lemma:20210820-l2}, we have
 $$\|u_\lambda\|_\infty\rightarrow 0~\hbox{as}~\lambda\rightarrow 0^+$$
 and so we can find some $\lambda_0 >0$ such that
 \beq\lab{eq:202109020e2}
 u_\lambda(x)\leq \|u_\lambda\|_\infty\leq s^*, \, \forall x\in \RN, \, \forall \lambda\in (0,\lambda_0].
 \eeq
 For any $\lambda^* \in (0, \lambda_0)$, we shall prove that $u_\lambda$ is continuous at $\lambda=\lambda^*$, namely that for any sequence $\lambda_n\rightarrow \lambda^*$, we have that $u_{\lambda_n}\rightarrow u_{\lambda^*}$ in $H^1(\RN)$.
 Without loss of generality, we may assume that
 $$\frac{\lambda^*}{2}\leq \lambda_n\leq \lambda_0, \forall n\in \mathbb{N}.$$
 Noting that $u_{\lambda_n} \in H^1(\RN)$ satisfies
 \beq\lab{eq:202109020e3}
 \|\nabla u_{\lambda_n}\|_2^2+\lambda_n\|u_{\lambda_n}\|_2^2=\int_{\RN}g(u_{\lambda_n})u_{\lambda_n}dx,
 \eeq
we obtain, combining \eqref{eq:202109020e3} with \eqref{eq:20210901-e11} and \eqref{eq:20210901-e12}-\eqref{eq:202109020e2}, that
 $$
 (\gamma N-1) \lambda_n \|u_{\lambda_n}\|_2^2\leq (1+\mid 2-N\mid\gamma)\|\nabla u_{\lambda_n}\|_2^2\leq (1+\mid 2-N\mid\gamma)N m_{\lambda_0}.
 $$
 Hence,
\begin{equation*}
\|u_{\lambda_n}\|_2^2\leq \frac{2(1+\mid 2-N\mid\gamma)N}{(\gamma N-1)\lambda^*}m_{\lambda_0} \quad \mbox{and} \quad \|\nabla u_{\lambda_n}\|_2^2=Nm_{\lambda_n}\leq Nm_{\lambda_0},
\end{equation*}
where we have used the fact that $ \lambda \mapsto m_{\lambda}$ is non-decreasing due to the mountain pass characterization \eqref{eq:20210901-e5}-\eqref{eq:20210901-e6}.
 So we conclude that $\{u_{\lambda_n}\}$ is bounded in $H^1(\RN)$. Now, by the compact embedding of $H_{rad}^{1}(\RN)$ into $ L^p(\RN), 2<p<2^*$ and using the equation of $u_n$, we can  prove that $\{u_{\lambda_n}\}$ is compact in $H^1(\RN)$. Hence, up to a subsequence, $u_{\lambda_n}\rightarrow u$ in $H^1(\RN)$, where $u\in H_{rad}^{1}(\RN)$ is a positive solution of \eqref{eq:NLS} with $\lambda=\lambda^*$. Then the uniqueness implies that $u=u_{\lambda^*}$. That is, the map $\lambda\mapsto u_\lambda$ is continuous at $\lambda=\lambda^*$.
\end{proof}

\brr\lab{remark:20210901-xr1}
Clearly the conclusion of Lemma \ref{cro:20210901-xc1} could be extended to all $\lambda >0$ if the uniqueness of positive solutions is known. However, note that
\begin{itemize}
\item[(i)] The problem of deriving conditions on $g(s)$, which insure that \eqref{eq:NLS} has a unique positive solution, has been extensively studied. The uniqueness is known to hold only for some specific class of nonlinearities. We refer to  \cite{Kwong1989, LewinRotaNodari2020, PucciSerrin1998} and the references therein in that direction.
\item[(ii)] We also note that the uniqueness of positive solution is in general not true. In \cite{JMI2013}, assuming $N=3,\lambda=1$,  $g(s)=\mid s\mid^{p-1}s+\mu \mid s\mid^{q-1}s$ with $1<q<3$, $p<5$ close to $5$ and $\mu$ large enough, the authors proved that \eqref{eq:NLS} possesses at least three different positive solutions.
\end{itemize}
\er

We now focus on the case of $N=1$. Then it is known, see \cite[Theorem 5]{BerestyckiLions1983} or \cite[Theorem 1.2]{JeanTanaka2003}, that for any $\lambda >0$ there exists a unique solution in $H^1(\mathbb{R})$ to \eqref{eq:NLS} and that this solution is, up to a translation of the origin, a positive even (radial) decreasing function.
The following result will be crucial in the proof of Theorem \ref{thm:20210822-th1} when $N=1$.

\bt\lab{thm:20210901-th1}
Let $N=1$ and assume that (G1)-(G2) hold. For any $\lambda>0$, let $u_\lambda$ denote  the unique solution of \eqref{eq:NLS} in $H^1(\mathbb{R})$. Then the map $\lambda\mapsto u_\lambda$ is  continuous from $(0,+\infty)$ to $H^1(\mathbb{R})$. In particular $\{(\lambda,u_\lambda): \lambda\in (0,+\infty)\}$ is connected.
\et
 \begin{proof}
 For any $\lambda_0>0$, we shall prove that $u_\lambda$ is continuous at $\lambda=\lambda_0$, namely that for any sequence $\lambda_n\rightarrow \lambda_0$, we have that $u_{\lambda_n}\rightarrow u_{\lambda_0}$ in $H^1(\mathbb{R})$. We claim that $\{u_{\lambda_n}\}$ is bounded in $H^1(\mathbb{R})$. The boundedness of
$\{\|\nabla u_{\lambda_n}\|_2^2 \}$ is a direct consequence of \eqref{eq:20210901-e12} and of the fact that $m \mapsto m_{\lambda}$ is non-decreasing. To prove the boundedness of $\{\| u_{\lambda_n}\|_2^2 \}$ we set, for any fixed $n$, $f(s)=-\lambda_n s+g(s)$ and recall, see \cite[Theorem 5]{BerestyckiLions1983}) that
$$\max_{x\in \mathbb{R}}u_{\lambda_n}(x)=u_{\lambda_n}(0)=s_0,$$
where $s_0 >0$ is the unique value for which $F(s)<0$ for $s\in (0,s_0)$ and $ F(s_0)=0.$
Here $F(s):= \int_0^s f(t)dt$.

Since $u_{\lambda_n}(x)\leq s_0,x\in \mathbb{R}$, we have that $F(u_{\lambda_n}(x))\leq 0,x\in \mathbb{R}$. Hence
\beq\lab{eq:20210901-we1}
\int_{\mathbb{R}}G(u_{\lambda_n})dx\leq \frac{\lambda_n}{2} \|u_{\lambda_n}\|_2^2.
\eeq
Recalling that  $u_{\lambda_n}$ must satisfy the Pohozaev identity, i.e.,
\beq\lab{eq:20210901-we2}
\lambda_n \|u_{\lambda_n}\|_2^2=\int_{\mathbb{R}}G(u_{\lambda_n})dx+\|\nabla u_{\lambda_n}\|_2^2,
\eeq
we deduce, combining  \eqref{eq:20210901-we1} and \eqref{eq:20210901-we2}, that
$$
\|u_{\lambda_n}\|_2^2\leq \frac{2}{\lambda_n} \|\nabla u_{\lambda_n}\|_2^2.
$$
Hence, $\{\|u_{\lambda_n}\|_2^2\}$ is also bounded. The sequence of solutions $\{u_n\}\subset H^1(\mathbb{R})$ is bounded. Recalling that for each $n \in \mathbb{N}$, $u_{\lambda_n}$ is a decreasing function,  and making use of
 \cite[Proposition 1.7.1]{Cazenave2003}, we deduce that $\{u_n\}$ is compact in $L^p(\mathbb{R}), \forall ~2<p\leq \infty$. Then one see that  $u_n\rightarrow u$ in $H^1(\mathbb{R})$ up to a subsequence, where $u\in H_{rad}^{1}(\mathbb{R})$ is a positive solution of \eqref{eq:NLS} with $\lambda=\lambda_0$. As in the proof of Lemma \ref{cro:20210901-xc1}, the uniqueness of positive solutions permits to conclude.
 \end{proof}

\section{Global branches of positive solutions}\label{sec:global}
\numberwithin{equation}{section}

When $N=1$, we have shown in Theorem \ref{thm:20210901-th1} just under (G1)-(G2) that, for any $\lambda \in (0, + \infty)$,  the unique positive solution $u_{\lambda} \in H^1(\mathbb{R})$ of \eqref{eq:NLS} lies on a global curve. In this section, under (G1)-(G3), we show that when $N \geq 2$,  the local curve, whose existence is established in Lemma  \ref{cro:20210901-xc1}, belongs to a connected component of the set of positive solutions, which contains a positive solution for every $\lambda \in (0, + \infty)$.
Let $$\mathcal{S}=\{(\lambda, u)\in (0, + \infty) \times H_{rad}^{1}(\RN)  : (\lambda, u)\;\hbox{solves \eqref{eq:NLS}}, \, u>0 \}.$$
We define $\widetilde{\mathcal{S}} \subset \mathcal{S}$ as the connected component of $\mathcal{S}$ containing the solutions ($\lambda, u_{\lambda})$ for $\lambda \in (0, \lambda_0)$. Denoting by $P_1: (0, +\infty) \times  H_{rad}^{1}(\RN)  \rightarrow (0, + \infty)$ the projection onto the $\lambda$-component, we shall prove that $P_1(\widetilde{\mathcal{S}}) = (0,+ \infty)$.
To reach this conclusion,  we  reformulate \eqref{eq:NLS} into a fixed point problem.
For $\lambda>0$ fixed, the  norm $\|u\|_\lambda$, defined by
$$\|u\|_\lambda:=\left(\|\nabla u\|_2^2+\lambda\|u\|_2^2\right)^{\frac{1}{2}},$$
is equivalent to the usual norm $\|u\|_{H^1(\RN)}$. The gradient of $J_{\lambda}$ with respect to $\langle \cdot, \cdot\rangle_\lambda$ can be computed as
$$\nabla J_\lambda (u)=u- \left(-\Delta +\lambda\right)^{-1} g(u) =: u -\mathbb{T}_\lambda(u).$$
We shall apply classical degree theory arguments to the equation
$$
\mathbb{T}_\lambda(u)=u, \, \lambda\in (0, + \infty), \, u\in H^1_{rad}(\RN).
$$
Some preliminaries are necessary.

\bl\lab{lemma:20210903-zl1}
Let $N\geq 2$ and assume that (G1)-(G2) hold.  The operator $\mathbb{T}_\lambda: H_{rad}^{1}(\RN)\rightarrow H_{rad}^{1}(\RN)$ is completely continuous.
\el
\begin{proof}
We need to prove that $\mathbb{T}_\lambda$ maps any bounded set into a pre-compact set. We shall work on $H_{rad}^{1}(\RN)$ with the norm $\|\cdot\|_\lambda$.
Let $\{u_n\}\subset H_{rad}^{1}(\RN)$ be a bounded sequence. By the compact embedding $H_{rad}^{1}(\RN)\hookrightarrow\hookrightarrow L^p(\RN), \, 2<p<2^*$, there exists a non-negative $u\in H_{rad}^{1}(\RN)$  such that, passing to a subsequence if necessary,
\beq\label{convergence}
u_n\rightharpoonup u~\hbox{in}~H^1(\RN)~\hbox{and}~u_n\rightarrow u~\hbox{in}~L^p(\RN), \, \forall \,  2<p<2^*.
\eeq
Let $\phi_n:=\mathbb{T}_\lambda(u_n)$, i.e.,
\beq\lab{eq:20210903-zhje1}
-\Delta \phi_n+\lambda\phi_n=g(u_n)~\hbox{in}~ H^1(\RN).
\eeq
When $N\geq 3$, under the assumptions (G1)-(G2), for any $\varepsilon >0$ there exists $C_{\varepsilon} < + \infty$ such that
\beq\lab{eq:20210903-zhje2}
g(s)\leq \varepsilon s+C_\varepsilon s^{q},\,  q=2^*-1, \, s\in (0, + \infty).
\eeq
Testing \eqref{eq:20210903-zhje1} with $\phi_n$, by the H\"older inequality and the critical Sobolev inequality, we have
$$
\|\phi_n\|_{\lambda}^{2}= \int_{\RN}g(u_n)\phi_ndx
\leq \varepsilon \|u_n\|_2 \|\phi_n\|_2+C_\varepsilon \|u_n\|_{2^*}^{2^*-1} \|\phi_n\|_{2^*}
\leq  C_{\varepsilon, M} \|\phi_n\|_{\lambda},
$$
which implies that $\{\phi_n\}\subset H_{rad}^{1}(\RN)$ is bounded.  Assuming that $\phi_n\rightharpoonup \phi$ in $H^1(\RN)$, it follows that $\phi\in H^1(\RN)$ is a non-negative weak solution to
\beq\lab{eq:20210903-zhje3}
-\Delta \phi+\lambda \phi= g(u), \quad u \in H^1_{rad}(\RN).
\eeq
Now, combining  \eqref{convergence} and Lemma \ref{20210822-wl1}-(ii), we deduce, in a standard way, that
\beq\lab{eq:20210903-zhje4}
\int_{\RN}g(u_n)\phi_n \, dx\rightarrow \int_{\RN}g(u)\phi \, dx.
\eeq
Then by \eqref{eq:20210903-zhje1}, \eqref{eq:20210903-zhje3}-\eqref{eq:20210903-zhje4}, we have that
$$\|\phi_n\|_{\lambda}^{2}\rightarrow \|\phi\|_{\lambda}^{2}~\hbox{as}~n\rightarrow +\infty,$$
which implies that $\phi_n\rightarrow \phi$ in $H^1(\RN)$.
When $N=2$,  we take $q$  in \eqref{eq:20210903-zhje2} satisfying $q=\max\{\alpha-1,\beta-1\}$.  We can still  write
by H\"older inequality and the Sobolev embedding,
$$ \|\phi_n\|_{\lambda}^{2}=\int_{\RN}g(u_n)\phi_ndx
\leq \varepsilon \|u_n\|_2 \|\phi_n\|_2+C_\varepsilon \|u_n\|_{2q}^{q} \|\phi_n\|_{2}
\leq  C_{\varepsilon, M} \|\phi_n\|_{H^1(\RN)},
$$
and the rest of the proof proceeds as when $N \geq 3$.
\end{proof}

We recall that local fixed point index $ind(\mathbb{T}_\lambda, u_\lambda)$ is defined as
$$ind(\mathbb{T}_\lambda, u_\lambda)=deg\left(id-\mathbb{T}_\lambda, N_\varepsilon(u_\lambda), 0\right),$$
where $\varepsilon>0$ is small, $N_\varepsilon$ denotes the $\varepsilon$-neighborhood in $H^1_{rad}(\RN)$, and  deg denotes the Leray-Schauder degree. It is well defined when $u_\lambda$ is an isolated fixed point of $\mathbb{T}_\lambda$ in $H^1_{rad}(\RN)$.
\bl\lab{prop:20210821-p1}
Let $N \geq 2$ and assume that (G1)-(G3) hold. Let $\lambda_0 >0$ be given by Lemma \ref{cro:20210901-xc1}. For $\lambda\in (0,\lambda_0)$, let $u_\lambda$ be the unique positive solution to \eqref{eq:NLS}.  Then it holds $ind(\mathbb{T}_{\lambda}, u_{\lambda})=-1, \, \forall \lambda\in (0,\lambda_0)$.
\el
\begin{proof}
The lemma follows directly from Theorem \ref{Hofer}. Let us show that the condition $(\Phi)$ is satisfied. Denote by $\gamma_1 \in \mathbb{R}$ the first eigenvalue of $J''_{\lambda}(u_{\lambda})$ and $v \in H_{rad}^1(\RN)$ an associated eigenfunction, namely
$$v - ( - \Delta + \lambda)^{-1}g'(u_{\lambda})v = \gamma_1 v.$$
We assume that $\gamma_1 \leq 0$ and define $\mu := (1- \gamma_1)^{-1} \in (0, 1]$ we obtain that $0 \neq v \in  H_{rad}^1(\RN)$ is such that
$$ - \Delta v + \lambda v = \mu g'(u_{\lambda})v.$$
By the choice of $\lambda_0$, \eqref{eq:202109020e1}-\eqref{eq:202109020e2} readily imply that the eigenvalue problem
$$-  \Delta v + \lambda v = \theta g'(u_{\lambda})v, \quad v \in  H_{rad}^1(\RN)$$
possesses a smallest positive eigenvalue $\theta_1 \in (0,1]$.
Furthermore,  since $\|u_\lambda\|_\infty\rightarrow 0$ as $\lambda\rightarrow 0^+$, for $\lambda>0$ small enough, it holds that $g'(u_\lambda(x))>0$ in $\RN$. Then it is standard to show that $\theta_1$ is a simple eigenvalue. Thus, since $\gamma_1 = 1-\frac{1}{\theta_1}$ the condition $(\Phi)$ holds.  Also, by Lemma \ref{lemma:MPL} we know that the unique critical point $u_{\lambda}$ is of mp-type. Finally, since $\{(\lambda, u_\lambda):\lambda\in(0,\lambda_0)\}$ is a curve, the conclusion holds for all $ \lambda\in (0,\lambda_0)$.
\end{proof}

\bl\lab{cro:20210822-c1}
Let $N \geq 2$ and assume that (G1)-(G3) hold. Then
$P_1(\widetilde{\mathcal{S}})= (0, \infty).$
\el
\begin{proof}

We assume by contradiction that there exists a $\overline{\lambda} > \lambda_0$ such that $P_1(\widetilde{\mathcal{S}}) \cap \{\overline{\lambda}\} = \emptyset.$  Let $a =  \frac{\lambda_0}{2}$ and $b > \overline{\lambda}.$ In view of Corollary \ref{cro:20210824-c1} and Lemma \ref{lemma:no-bifurcation}, there exist two positive constants $C_1(a,b) < C_2(a,b)$ such that
\beq\label{inside}
 0 < C_1(a,b) < \|u\|_{H^{1}(\RN)} < C_2(a,b), \quad \mbox{for any} \, u \in \mathcal{U}_{a}^{b}.
\eeq
We define
$$\Omega = \{u \in H^{1}(\RN): \,  C_1(a,b) < \|u\|_{H^{1}(\RN)} < C_2(a,b)\},$$
$$ \Phi (\lambda, u) = u - T_{\lambda}(u), \, u \in H^{1}(\RN)$$
and
$$\Sigma = \{ (\lambda, u) \in [a,b] \times \overline{\Omega} : \Phi(\lambda,u) =0 \}.$$
We also use the notation $\Sigma_{\lambda}$ for the $\lambda$-slice of $\Sigma$, i.e.
$$\Sigma_{\lambda}= \{u \in \overline{\Omega} : (\lambda, u) \in \Sigma \}.$$
We note that
$$\Phi(\lambda, u) = u - T_{\lambda}(u) \neq 0, \quad \forall (\lambda, u) \in [a,b] \times \partial \Omega.$$
Indeed, if $\Phi(\lambda,u) =0$, then $u$ is positive by the maximum principle. Thus $u \in \mathcal{U}_{a}^{b}$ and \eqref{inside} holds. At this point, we apply
Proposition \ref{Proposition-degree} with $X = H_{rad}^{1}(\RN)$ and $\Omega, \Phi$ defined as above. Note that the assumption
$deg (\Phi(a,\cdot), \Omega, 0) \neq 0$  holds because we know from Lemma \ref{prop:20210821-p1} that
$deg(\Phi(a,\cdot), \Omega, 0) = ind(\mathbb{T}_{\lambda}, u_{\lambda}) =-1$. We conclude that the compact connected set $\mathcal{C} \subset \Sigma$  provided by Proposition \ref{Proposition-degree}  satisfies
$$\mathcal{C} \cap (\{a\} \times \Sigma_a) \neq \emptyset  \quad \mbox{and} \quad \mathcal{C} \cap (\{b\} \times \Sigma_b) \neq \emptyset.$$
Since necessarily $ \mathcal{C}$ coincides with $\widetilde{\mathcal{S}}$, this is in contradiction with our assumption that $P_1(\widetilde{\mathcal{S}}) \cap \{\bar{\lambda}\} = \emptyset $. The lemma is proved.
\end{proof}

\section{Application to existence, non-existence and multiplicity of positive normalized solutions.}\label{sec:application}
\numberwithin{equation}{section}

In this section we give the proof of our main result.

\noindent{\bf Proof of Theorem \ref{thm:20210822-th1}:}
Let us introduce the function
\beq\lab{eq:rho}
  \rho: \mathcal{S}\rightarrow (0,+\infty),\quad (\lambda,u)\mapsto \|u\|_2^2.
\eeq

{\bf(i)}
By Lemma \ref{cro:20210822-c1} for $N\geq 2$ and Theorem \ref{thm:20210901-th1} for $N=1$, there exists a connected set $\widetilde{\mathcal{S}}\subset \mathcal{S}$ with $P_1(\widetilde{\mathcal{S}})=(0,+\infty)$.
Recalling Theorem \ref{cro:20210824-xc1}, there exist $(\lambda_n, u_n)\subset \widetilde{\mathcal{S}}$ with $\lambda_n\rightarrow 0^+$ and $\|u_n\|_2^2\rightarrow 0$. Similarly, there exist $(\lambda'_n, u'_n)\subset \widetilde{\mathcal{S}}$ with $\lambda_n\rightarrow +\infty$ and $\|u\|_2^2\rightarrow +\infty$. Since $\widetilde{\mathcal{S}}$ is connected, it follows that $\rho$ is onto.

{\bf (vi)} As in the proof of (i), we have that $\rho(\widetilde{\mathcal{S}})=(0,+\infty)$, and the conclusion follows from Theorem \ref{cro:20210824-xc1}.

{\bf (ii)} By $$\rho(\mathcal{S})\supset \rho(\widetilde{\mathcal{S}})\supset (\min\{\|U\|_2^2, \|V\|_2^2\}, \max\{\|U\|_2^2, \|V\|_2^2\}),$$ we see that for any $a\in (\min\{\|U\|_2^2, \|V\|_2^2\}, \max\{\|U\|_2^2, \|V\|_2^2\})$,  \eqref{eq:NLS} possesses at least one normalized solution $(\lambda, u_\lambda)$ with $\lambda>0$ and $0<u_\lambda\in H_{rad}^{1}(\RN)$.
Recalling Theorem \ref{thm:loc-uniqueness}, \eqref{eq:NLS} has a unique solution $u_\lambda >0$ for $\lambda >0$ small or large enough. So for any $\delta\in (0, \min\{\|U\|_2^2, \|V\|_2^2\})$,  there exists some $\Lambda_1>0$ small enough and $\Lambda_2>0$ large enough such that
$$\begin{cases}
\rho(\lambda, u_\lambda)\in (\|U\|_2^2-\delta, \|U\|_2^2+\delta), &\forall \lambda\in (0,\Lambda_1),\\
\rho(\lambda, u_\lambda)\in (\|V\|_2^2-\delta, \|V\|_2^2+\delta), &\forall \lambda\in (\Lambda_2,+\infty).
\end{cases}
$$
On the other hand, by Corollary \ref{cro:20210824-c1} and Lemma \ref{lemma:no-bifurcation}, we can find some $M_1,M_2>0$ such that
$$\left\{\|u\|_2^2: u\in\mathcal{U}_{\Lambda_1}^{\Lambda_2}\right\}\subset [M_1,M_2].$$
Hence, we can take
$$a_1=\min\{\|U\|_2^2-\delta,\|V\|_2^2-\delta, M_1\}~\hbox{and}~a_2=\max\{\|U\|_2^2+\delta,\|V\|_2^2+\delta,M_2\}.$$
Then for any $a\in (0,a_1)\cup (a_2,+\infty)$, \eqref{eq:NLS}-\eqref{eq:norm} has no positive normalized solution.

{\bf (iii)} We only prove (iii-1).  By $$\rho(\mathcal{S})\supset \rho(\widetilde{\mathcal{S}})\supset (0,  \|V\|_2^2),$$ applying a similar argument as (ii), we can prove that \eqref{eq:NLS}-\eqref{eq:norm} has at least one normalized solution $(\lambda, u_\lambda)$ with $\lambda>0$ and $0<u_\lambda\in H_{rad}^{1}(\RN)$ if $0<a<\|V\|_2^2$. And there exists some $a^*\geq \|V\|_2^2$ such that \eqref{eq:NLS}-\eqref{eq:norm} has no positive normalized solution provided $a>a^*$. This finishs the proof of (iii-1). The case of (iii-2) can be proved similarly.

{\bf (v)} We only prove (v-1). By $$\rho(\mathcal{S})\supset \rho(\widetilde{\mathcal{S}})\supset (0,  \|U\|_2^2),$$
we can prove it in a similar way to (iii-1).

{\bf (iv)} We only prove (iv-1). In such  case, we have $\|u_\lambda\|_2^2\rightarrow 0$ both as $\lambda\rightarrow 0^+$ and as $\lambda\rightarrow +\infty$.
Define
$$a^*:=\max_{\widetilde{\mathcal{S}}}\rho(\lambda, u_\lambda).$$
Then there exists some $\lambda^*>0$ with $\|u_{\lambda^*}\|_2^2=a^*$. Then for any $a\in (0,a^*)$, there exist some $\lambda_1\in (0,\lambda^*)$ and $\lambda_2\in (\lambda^*, +\infty)$ such that
$$\rho(\lambda_1, u_{\lambda_1})=\rho(\lambda_2, u_{\lambda_2})=a.$$
That is, for any $a\in (0,a^*)$, \eqref{eq:NLS}-\eqref{eq:norm} has at least two different normalized solutions $(\lambda_i,u_i)$ with $\lambda_i>0$ and $0<u_i\in H_{rad}^{1}(\RN), i=1,2$.
Furthermore, applying a similar argument as in (ii), we can find some $M\geq a^*$ such that
 \eqref{eq:NLS}-\eqref{eq:norm} has no positive normalized solution provided $a>M$.  The case of (iv-2) can be proved similarly. We only note that in such a case, $\|u_\lambda\|_2^2\rightarrow +\infty$ both as $\lambda\rightarrow 0^+$ and $\lambda\rightarrow +\infty$.
\hfill$\Box$. 
\end{document}